\documentclass{amsart}
\usepackage{amssymb,mathrsfs}
\usepackage{color,umoline}
\usepackage[dvipsnames]{xcolor}
\usepackage{graphicx}
\usepackage{enumitem}
\usepackage[font=small,labelfont=bf]{caption}
\usepackage{tikz, caption}
\usepackage{relsize}
\usepackage[mathscr]{eucal}
\usepackage{dsfont}
\usepackage{verbatim}
\usepackage[draft]{todonotes}
\usepackage[
 citecolor=black,  
  hypertexnames=false]{hyperref}
\usepackage[dvipsnames]{xcolor}
\usetikzlibrary{arrows.meta}
\usepackage{subfig}
\usepackage{verbatim}
\usepackage{kotex}
\usepackage[utf8]{inputenc}

\def\wt#1{\widetilde{#1}}

\newcommand{\N}{\mathbb{N}}
\newcommand{\R}{\mathbb{R}}
\newcommand{\Z}{\mathbb{Z}}

\newcommand{\supp}{\operatorname{supp}}
\newcommand{\dist}{\operatorname{dist}}

\newtheorem{thm}{Theorem}[section]
\newtheorem{defn}[thm]{Definition}
\newtheorem{prop}[thm]{Proposition}
\newtheorem{cor}[thm]{Corollary}
\newtheorem{lem}[thm]{Lemma}
\newtheorem{conj}[thm]{Conjecture}
\numberwithin{equation}{section}
\newtheorem{rem}{Remark}

\newcommand{\inp}[2]{\langle #1, #2\rangle}

\newcommand{\sxyl}{2^{-l}s+\sxy}

\newcommand{\vepc}{\varepsilon_{\!\circ}}

\newcommand{\fP}{\mathfrak P}

\renewcommand{\sxyl}{S_c^l}

\newcommand{\cP}{\phi_1}
\newcommand{\dpq}{\delta(p,q)}

\newcommand{\che}{\chi_{E_\lambda}}

\newcommand{\ppqpair}{$(1/p, 1/q)$}

\newcommand{\ppq}{(1/p,1/q)}
\newcommand{\Be}{\begin{equation}}
\newcommand{\Ee}{\end{equation}}

\newcommand{\cQ}{\mathcal Q}
\newcommand{\cD}{\mathcal D}

\newcommand{\sxyls}{\sxyl(x,y,s)} 

\newcommand{\eval}{2\mathbb N_0+d}

\newcommand{\rp}{1/p}
\renewcommand{\rq}{1/q}

\newcommand{\sxylt}{\sxyl(x,y, t)}

\newcommand{\chplp}{\chi_{\lambda, \mu}}

\newcommand{\diam}{{\rm diam }}

\newcommand{\proj}{{\mathlarger \wp}_k}
\newcommand{\wproj}{\widetilde{\mathlarger \wp}_k}

\newcommand{\B}{\mathbb B}

\newcommand{\tx}{\textstyle}

\begin{document}

\author[Jeong]{Eunhee Jeong}
\address[Jeong]{Department of Mathematics Education, and  Institute of Pure and Applied Mathematics, Jeonbuk National University, Jeonju 54896, Republic of Korea}
\email{eunhee@jbnu.ac.kr}

\author[Lee]{Sanghyuk Lee}
\address[Lee]{Department of Mathematical Sciences and RIM, Seoul National University, Seoul 08826, Republic of  Korea}
\email{shklee@snu.ac.kr}

\author[Ryu]{Jaehyeon Ryu}
\address[Ryu]{School of Mathematics, Korea Institute for Advanced Study, Seoul
02455, Republic of Korea} 
\email{jhryu@kias.re.kr}


\title[Bounds on the Hermite spectral projection operator]{Bounds on the Hermite spectral \\ projection operator}

\begin{abstract} 
We study  $L^p$--$L^q$ bounds on  the spectral projection operator $\Pi_\lambda$  associated to the Hermite operator $H=|x|^2-\Delta$ in $\mathbb R^d$.   We are mainly  concerned with a localized operator $\chi_E\Pi_\lambda\chi_E$ for a  subset $E\subset\mathbb R^d$  and undertake the task  of  characterizing the sharp $L^p$--$L^q$ bounds. We obtain  sharp bounds in extended ranges of $p,q$. First, we  provide a complete characterization of  the sharp $L^p$--$L^q$ bounds when $E$ is away from $\sqrt{\lambda}\mathbb S^{d-1}$. Secondly, we obtain the sharp bounds as the set $E$ gets close to $\sqrt\lambda\mathbb S^{d-1}$. 
Thirdly, we extend the range of $p,q$ for which the operator $\Pi_\lambda$ is uniformly bounded from $L^p(\mathbb R^d)$ to $L^q(\mathbb R^d)$.
\end{abstract}

\maketitle

\section{Introduction}  

Let $H$ denote the Hermite operator $-\Delta+|x|^2$ in $\R^d$, $d\ge 2$. The  operator $H$ has a discrete spectrum $\lambda\in 2\mathbb N_0+d$, where $\mathbb N_0=\mathbb N\cup \{0\}$.  For $\alpha\in \mathbb N_0^d$, let $\Phi_\alpha$ be the $L^2$--normalized Hermite function which is  
an  eigenfunction of $H$ with eigenvalue $2(\alpha_1+\dots+\alpha_d)+d$. The set  $\{\Phi_\alpha: \alpha\in \mathbb N_0^d\}$  forms an orthonormal basis in $L^2$.  We consider the spectral  projection operator  
\begin{align*}
    \Pi_\lambda f = \underset{\alpha:d+2|\alpha|=\lambda}{\sum}\langle f,\Phi_\alpha \rangle \Phi_\alpha,
\end{align*}
which is the orthogonal projection  to the vector space  spanned by eigenfunctions with  the eigenvalue $\lambda$. Then,  $f=\sum_{\lambda\in 2\mathbb N_0+d}\Pi_\lambda f$ for  $f\in L^2$.

$L^p$--$L^q$ bounds on the spectral projection operators associated to differential operators  have been studied by various authors (see, for example, \cite{sogge-1986, sogge-1988, SZ, KR, KL18, JLR_twist}). 
Let $\|T\|_{p\to q}$ denote the operator norm  of  an operator $T$ from $L^p$ to $L^q$.   Concerning the Hermite  operator, the bounds  
\begin{equation}\label{rstest}
\|\chi_E \Pi_\lambda \chi_E \|_{p\to q} \le \mathrm B(\lambda, p,q)
\end{equation}
with suitable subsets $E\subset \mathbb R^d$  has been  of interest and studied by some authors. 
    The  estimates are related to Bochner-Riesz summability of the Hermite expansion \cite{K94,Th98} and the  unique continuation properties for the parabolic operators \cite{e00, ev01}. 
   When $E=\R^d$, we call \eqref{rstest} a global estimate. In such a case, the bound \eqref{rstest} with 
$p=2$ and $2\le q\le \infty$ was studied  by Thangavelu \cite{Th93}, Karadzhov \cite{K94}, and Koch--Tataru \cite{T05}. Especially, Koch and Tataru obtained the optimal $L^2$--$L^q$ bound  for $2\le q\le \infty$ except $q=2(d+3)/(d+1)$. Recently, the missing endpoint estimate was proved by the authors \cite{JLR_endpoint} for $d\ge 3$.

In this paper, we are mainly concerned with  local estimates for the projection  $\Pi_\lambda$, i.e.,  the estimate \eqref{rstest} with  bounded sets $E$.   As shown in the earlier works \cite{Th93, K94,  Th98,  T05},  $\Pi_\lambda$ exhibits different behaviors across the sphere $\sqrt\lambda\mathbb  S^{d-1}=\{x: |x|=\sqrt\lambda\}$.

\subsubsection*{Estimate over $\sqrt\lambda \B$} We first consider the case $E$ is distanced away from $\sqrt\lambda\mathbb  S^{d-1}$.  
Let $\mathbb B=\{x: |x|<1/2\}$.    In view of the transplantation result  due to Kenig, Stanton, and Tomas \cite{KST}, it seems to be plausible to expect 
that the $L^p$--$L^q$ bounds on $\chi_{\B}\Pi_\lambda\chi_{\B}$  have similar behaviors as those on 
\begin{equation}
\label{def-pk}
    \proj f =\frac{1}{(2\pi)^d} \int_{k-1\le|\xi|^2\le k} e^{i x\cdot \xi} \widehat f(\xi) d\xi.
\end{equation} 
The sharp bounds on $\proj$  in terms of $k$ (Proposition \ref{laplacian} below) can be deduced by  a rescaling argument and the boundedness of the restriction-extension operator  $f\to ( \widehat f\,|_{\mathbb S^{d-1}})^\vee $ (see Theorem \ref{rr*}), which is closely related to the Bochner-Riesz operator of negative orders (see Section \ref{transplantation} for further details).   Our first result (Theorem \ref{thm-locest} below) demonstrates validity of  the aforementioned heuristics, and consequently  
provides a complete characterization of $L^p$--$L^q$ bounds on  $\chi_{\sqrt{\lambda}\B}\Pi_\lambda\chi_{\sqrt{\lambda} \B}$. 

\newcommand{\sq}{{\mathlarger{\square}}}

To state our result, we need some notations. For $X=(a, b)\in  \sq:= [1/2,1]\times[0,1/2]$, we denote $X'=(1-b, 1-a)$.  Likewise, we define
$\mathfrak Z'=\{ X': X\in \mathfrak Z\}$ for a set $\mathfrak Z\subset \sq$. If $X, Y\in \sq$ and $X\neq  Y,$   $[X,Y]$ and $(X,Y)$  denote the closed and open line segments connecting $X$ and $Y,$ respectively. Similarly, the half open  line segments  $(X, Y]$, $[X, Y)$ are  defined.  
Finally, if $X_1, \dots, X_k\in \sq$, by $[X_1, \dots, X_k]$ we denote the convex hull of  $X_1, \dots, X_k$. 

\begin{defn}
Let $\mathfrak A=\mathfrak A(d),$ $ \mathfrak C=\mathfrak C(d)$, and $\mathfrak D=\mathfrak D(d)\in \sq$ be  given    by
\begin{align*}
    &\mathfrak A=\left( \dfrac{d+3}{2(d+1)}, \  \dfrac12\right), 
        \  \mathfrak C=\left(\dfrac{d^2+4d-1}{2d(d+1)},\dfrac{d-1}{2d}\right),
   \ \mathfrak D=\left(1,\frac{d-1}{2d}\right).
\end{align*}
Let $\mathcal R_1=[(\frac12, \frac12),  \mathfrak A,  \mathfrak C,  \mathfrak C', \mathfrak A']\setminus \{\mathfrak C, \mathfrak C'\}$,   $\mathcal R_2=[\mathfrak A , (1,1/2), \mathfrak D, \mathfrak C]\setminus [\mathfrak C, \mathfrak D]$, and  $\mathcal R_3=[\mathfrak C, \mathfrak D, (1,0), \mathfrak D', \mathfrak C']\setminus ([\mathfrak C, \mathfrak D]\cup [\mathfrak C', \mathfrak D'])$. $($See Figure \ref{local-type}$)$.
\end{defn}

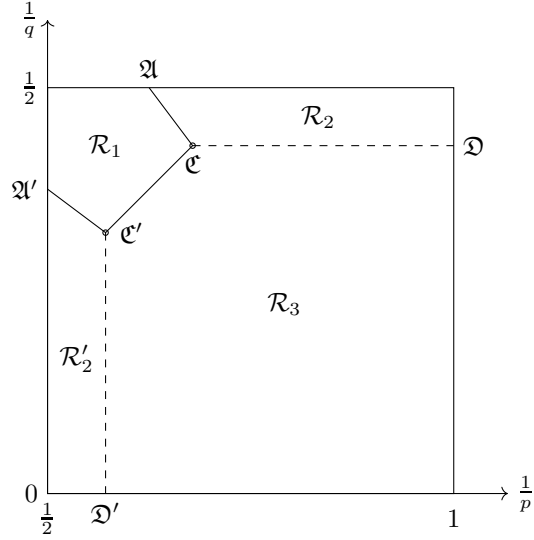
\begin{figure}
\begin{tikzpicture}[scale=0.9]
\draw[<->] (0,7) node[left]{$\frac1q$}--(0,0)--(6.8,0)node[right]{$\frac1p$};
\draw (0,0) rectangle (6,6);
\draw (0,6)--(0,4.5)--(6/7,27/7) --(15/7,36/7)--(1.5,6);
\draw[dashed] (6/7,27/7)--(6/7,0); \draw[dashed] (15/7,36/7)--(6,36/7); 
\node[left] at (0,4.5) {$\mathfrak A'$}; \node[above] at (1.5,6) {$ \mathfrak A$}; 
\node[below] at (6/7,0) {$\mathfrak D'$}; \node[right] at (6,36/7) {$\mathfrak D$}; 
\node[] at (6/7,36/7) {$\mathcal R_1$};
\node[] at (3/7,14/7) {$\mathcal R_2'$}; \node[] at (28/7,39/7) {$\mathcal R_2$};
\node[above] at (3.5,2.5) {$\mathcal R_3$};
\draw (6/7,27/7) circle [radius=0.04]; \node[right] at (6.5/7,27/7) {$\mathfrak C'$};
\draw (15/7,36/7) circle [radius=0.04]; \node[below] at (15/7,36/7) {$\mathfrak C$}; 
\node[left] at (0,0) {$0$};\node[below] at (0,0) {$\frac12$};
\node[left] at (0,6) {$\frac12$};\node[below] at (6,-0.1) {$1$};
\end{tikzpicture}
\caption{The points $\mathfrak A,$ $\mathfrak C $, $\mathfrak D$, and the regions $\mathcal R_1$, $\mathcal R_2,$ $\mathcal R_3$. }
\label{local-type}
\end{figure}

For $(1/p,1/q)\in \sq$,  we define the exponent  $\beta(p,q)$ by setting\footnote{Note $ \beta(p,q)=\max \big(-\tfrac12\delta(p,q),  -1+\tfrac{d}{2}\delta(p,q), -\tfrac{d+1}{2}+\tfrac{d}{2}\big(\frac1p+\frac1q\big), \tfrac{d-1}{2}-\tfrac{d}{2}\big(\frac1p+\frac1q\big)\big). $}
\[
\beta(p,q)=
         \begin{cases} 
                 \  -\frac12\delta(p,q),    & \ \big(\frac1p,\frac1q\big)\in\mathcal R_1,  
                                              \\  
                 \ \frac{d}{2}\big(\frac1p+\frac1q\big) -\frac{d+1}{2}, & \ \big(\frac1p,\frac1q\big)\in\mathcal R_2, 
                    \\
               \     \frac{d-1}{2}-\frac{d}{2}\big(\frac1p+\frac1q\big),& \ \big(\frac1p,\frac1q\big)\in\mathcal R_2',
                   \\
               \    \frac{d}{2}\delta(p,q) -1,           & \ \big(\frac1p,\frac1q\big)\in {\mathcal R}_3\cup [\mathfrak C, \mathfrak D]\cup [\mathfrak C', \mathfrak D'] .
\end{cases} 
\]
Here $\delta(p,q)=1/p-1/q.$ For a given set $E\subset \mathbb R^d$ we denote $E_\lambda=\{ \sqrt \lambda x: x\in E\}$.

\begin{thm}\label{thm-locest}  Let $d\ge 2$ and $(1/p,1/q)\in \sq$.   Then,  we have 
\begin{align}\label{est-loc0}    \|\chi_{ \B_\lambda}\Pi_\lambda \chi_{\B_\lambda}\|_{p\to q}\lesssim    \lambda^{\beta(p,q)}
\end{align}
if and only if  $(1/p,1/q)\not\in [\mathfrak C,\mathfrak  D]\cup[\mathfrak C', \mathfrak  D']$. 
Moreover, we have
\begin{enumerate}
\item[$\!(i)$] $  \|\chi_{ \B_\lambda}\Pi_\lambda \chi_{\B_\lambda}\|_{L^p\to L^{q,\infty}}\lesssim \lambda^{\beta(p,q)}$   if $(1/p,1/q)\in (\mathfrak C, \mathfrak D]\cup (\mathfrak C', \mathfrak D']$,
\item[$(ii)$] $\|\chi_{ \B_\lambda}\Pi_\lambda \chi_{\B_\lambda}\|_{ L^{p,1}\to L^{q,\infty}} \lesssim   \lambda^{\beta(p,q)}$   if $(1/p,1/q)=\mathfrak C$ or $\mathfrak C'$.
\end{enumerate}
\end{thm}

 Here $\| \cdot \|_{L^{p,r}\to L^{q,s}}$ denotes the operator norm from the Lorentz space $L^{p,r}$ to $L^{q,s}$ (e.g., see \cite{St71}). 
 Remarkably,  the estimates in Theorem \ref{thm-locest} are  sharp. More precisely, 
 by Theorem \ref{thm-locest} and Proposition \ref{lower-mu} below  we have 
\[ \| \chi_{ \B_\lambda}\Pi_\lambda \chi_{\B_\lambda}\|_{p\to q}\sim \lambda^{\beta(p,q)} \]
for $(1/p,1/q)\in\sq\setminus([\mathfrak C,\mathfrak  D]\cup[\mathfrak C', \mathfrak  D'])$.
 When $p=2$ (equivalently, $q=p'$, or $q=2$)  the  sharp $L^p$--$L^q$ (local) bounds \eqref{est-loc0} were previously obtained  (\cite{K94, Th98, T05}).  However, we emphasize that 
the sharp bounds for other $p,q$ are not generally accessible  by mere interpolation between the  previously known bounds   due to change of the regimes  (see Figure \ref{local-type}).  
As alluded above, there is a strong resemblance between the local estimate for $\Pi_\lambda$ (Theorem \ref{thm-locest}) and the global estimate for ${\mathlarger \wp}_k$ (Corollary \ref{laplacian}).  For some special cases   the local estimates in Theorem \ref{thm-locest}  imply  those  in Corollary \ref{laplacian} (see Lemma \ref{implication}). 

The implication in Lemma \ref{implication} remains valid while $L^{q}$ is replaced  by Lorentz spaces $L^{q,\infty}$ as long as $q>1$. So, it is not possible to strengthen the weak type estimates  in Theorem \ref{thm-locest}  ($(i)$)  by replacing  $L^{q,\infty}$ with the smaller space $L^{q,r}$, $r<\infty$, because the same  is true for  $f\to ( \widehat f\, |_{ \mathbb S^{d-1}})^\vee$ (see Theorem \ref{rr*}).

\subsubsection*{Estimate near the sphere $\sqrt{\lambda}\mathbb S^{d-1}$}
As shown in \cite{Th87, KST},  $\Pi_\lambda$ exhibits different behaviors when the input functions are supported near (equivalently,  $L^q$ integration is taken over) the set $\sqrt{\lambda}\mathbb S^{d-1}$.   This naturally  leads  to considering  a localization getting close to  the sphere $\sqrt{\lambda}\mathbb S^{d-1}$. 
To do this, for   $ \mu \in \{ 2^k:k\in \mathbb Z\}$, set
\[
A_{\mu}=\big\{x:   (1- |x|)\in [2^{-1}\mu, \mu]\big\}, \quad   A_{\lambda, \mu}=\big\{x:   \lambda^{-\frac12} x\in A_{\mu} \big\}.
\]
We also denote
\[ \chi_\mu = \chi_{A_\mu}, \quad \chi_{\lambda, \mu}=\chi_{A_{\lambda,\mu}}.\]

\begin{figure}
\centering
\begin{minipage}{.4\textwidth}
  \centering
 \begin{tikzpicture}[scale=0.6]
\draw[<->] (0,7) node[left]{$\frac1q$}--(0,0)--(6.8,0)node[right]{$\frac1p$};
\filldraw[fill=gray!30] (0,6)--(0,2)--(4,2)--(4,6);
\filldraw[fill=gray!30] (0,2)--(0,0)--(3,0);
\filldraw[fill=gray!30] (4,6)--(6,6)--(6,3);
\filldraw[fill=gray!30] (4,2)--(3,0)--(6,0)--(6,3);
\draw (0,0) rectangle (6,6);
\draw (0,6)--(0,2)--(4,2)--(4,6); 
\draw (0,2)--(3,0); 
\draw (4,6)--(6,3); 
\draw (4,2)--(3,0); \draw (4,2)--(6,3); 
\node[left] at (0,2) {$\mathfrak A'$}; \node[above] at (4,6) {$ \mathfrak A$}; 
\draw (3,0) circle [radius=0.07]; \node[below] at (3,0) {$\mathfrak D'$};
\draw (6,3) circle [radius=0.07]; \node[right] at (6,3) {$\mathfrak D$}; 
\node[] at (2,4) {$\mathfrak L_1$};
\node[] at (16/3,5) {$\mathfrak L_2$}; \node[] at (1,2/3) {$\mathfrak L_2'$};
\node[] at (5,1) {$\mathfrak L_3$};
\node[anchor=north west] at (3.9,2.1) {$\mathfrak G$}; 
\node[left] at (0,0) {$0$};\node[below] at (0,0) {$\frac12$};
\node[left] at (0,6) {$\frac12$};\node[below] at (6,0) {$1$};
\end{tikzpicture}
\caption{ $d=2$. }
\label{xyx}
\end{minipage}
\hspace{30pt}
\begin{minipage}{.4\textwidth}
  \centering
 \begin{tikzpicture}[scale=0.6]
 
\draw[<->] (0,7) node[left]{$\frac1q$}--(0,0)--(6.8,0)node[right]{$\frac1p$};
\filldraw[fill=gray!30] (0,6)--(3,6)--(18/5,18/5)--(12/5,12/5)--(0,3);
\filldraw[fill=gray!30] (0,3)--(0,0)--(2,0);
\filldraw[fill=gray!30] (3,6)--(6,6)--(6,4);
\filldraw[fill=gray!30] (12/5,12/5)--(18/5,18/5)--(6,4)--(6,0)--(2,0);
\draw (0,0) rectangle (6,6);
\draw (0,6)--(0,3)--(12/5,12/5) --(18/5,18/5)--(3,6); 
\draw (0,3)--(2,0); 
\draw (3,6)--(6,4); 
\draw (12/5,12/5)--(2,0); \draw(18/5,18/5)--(6,4); 
\node[left] at (0,3) {$\mathfrak A'$}; \node[above] at (3,6) {$ \mathfrak A$}; 
\draw (2,0) circle [radius=0.07]; \node[below] at (2,0) {$\mathfrak D'$};
\draw (6,4) circle [radius=0.07]; \node[right] at (6,4) {$\mathfrak D$}; 
\node[] at (1.5,4.5) {$\mathfrak L_1$};
\node[] at (5,16/3) {$\mathfrak L_2$}; \node[] at (2/3,1) {$\mathfrak L_2'$};
\node[] at (4.2,1.8) {$\mathfrak L_3$};
\draw (12/5,12/5) circle [radius=0.07]; \node[right] at (12/5,12/5) {$\mathfrak G'$};
\draw (18/5,18/5) circle [radius=0.07]; \node[below] at (18/5,18/5) {$\mathfrak G$}; 
\node[left] at (0,0) {$0$};\node[below] at (0,0) {$\frac12$};
\node[left] at (0,6) {$\frac12$};\node[below] at (6,0) {$1$};
\end{tikzpicture}
\caption{$d\ge 3$. }
\label{xyxz}
\end{minipage}
\end{figure}

 To obtain the  sharp (global) $L^2$--$L^q$ estimate with $q\ge2$, Koch and Tataru \cite{T05} considered the localized operator $\chi_{\lambda,\mu}\Pi_\lambda$. They showed  
\begin{equation}  
\label{kt-nu}
 \|\chplp\Pi_\lambda \|_{2\to q}\sim 
                                      \begin{cases}
\lambda^{-\frac12\delta(2,q)}\mu^{\frac14-\frac{d+3}{4}\delta(2,q)}, &   \ \, \, 2\le q\le \frac{2(d+1)}{d-1},
\\
(\lambda\mu)^{-\frac12+\frac d2\delta(2,q)}, &   \   \frac{2(d+1)}{d-1} \le q\le \infty 
                                      \end{cases}
\end{equation}
for $\lambda^{-\frac23}\le\mu\le 1/4$ (see \cite[Theorem 3]{T05}).  In fact, a slightly different form of  weighted $L^2$ estimate  was shown but 
 the result  is essentially equivalent to \eqref{kt-nu}.  Since the Hermite functions decay exponentially outside the ball $B(0, \sqrt \lambda)$, the contribution from  $(1-\chi_{B(0, \sqrt \lambda)} )\Pi_\lambda$ 
is less significant. In fact, if   $\chplp$ in  \eqref{kt-nu} is replaced by the characteristic function of $A_{\lambda,\mu}^-:=\{\sqrt \lambda x:  |x|-1\in[2^{-1}\mu,\mu] \}$, 
similar but stronger estimates  can be shown.  By duality, the estimate \eqref{kt-nu}  is equivalent to 
\begin{equation}  
\label{ktdia}
 \|\chplp\Pi_\lambda\chplp \|_{q'\to q}\sim 
                                      \begin{cases}
\lambda^{-\frac12\delta(q',q)}\mu^{\frac14-\frac{d+3}{4}\delta(q',q)}, &   \  \, \, 2\le q\le \frac{2(d+1)}{d-1},
\\
(\lambda\mu)^{-1+\frac d2\delta(q',q)}, &   \   \frac{2(d+1)}{d-1} \le q\le \infty.
                                      \end{cases}
\end{equation}

 Our second result extends the estimate \eqref{ktdia} to $(p,q)$ other than $(q',q)$. 
We set
\begin{align*}
    \gamma(p,q)=\begin{cases} 
                 \  \frac12-\frac{d+3}{4}\delta(p,q),    & \  \  \big(\frac1p,\frac1q\big)\in\mathcal R_1,  
                                              \\  
                 \ d\big(\frac{1}{2p}+\frac1q\big)-\frac{3d+1}{4}, & \ \  \big(\frac1p,\frac1q\big)\in\mathcal R_2, 
                    \\
               \     \frac{3d-1}{4}-d\big(\frac1p+\frac{1}{2q}\big),& \ \  \big(\frac1p,\frac1q\big)\in\mathcal R_2',
                   \\
               \    \frac{d}{2}\delta(p,q) -1,           &  \ \ \big(\frac1p,\frac1q\big)\in  {\mathcal R}_3\cup [\mathfrak C, \mathfrak D]\cup [\mathfrak C', \mathfrak D'],
\end{cases}
\end{align*}
 for $(1/p,1/q)\in \sq$. We consider the estimate 
 \begin{align}\label{est-ann}
    \|\chi_{\lambda,\mu} \Pi_\lambda\chi_{\lambda,\mu}\|_{p\to q}\le C\lambda^{\beta(p,q)}\mu^{\gamma(p,q)},
\end{align}
which coincides with \eqref{ktdia} when $p=q'$.   It is not difficult to show that the exponent in \eqref{est-ann} can not be improved to any better one (see Proposition \ref{lower-mu}) up to a constant.   It seems to be plausible to expect that the next holds true.

 \begin{conj}\label{conjecture} For  $(1/p,1/q)\in \sq$, the estimate \eqref{est-ann} holds.  \end{conj}

We partially verify  Conjecture \ref{conjecture}. In order to state our result we need additional notations. 
\begin{defn}
Let 
\begin{align*}
    \mathfrak G=\left(\frac{2d^2+7d-7}{2(2d-1)(d+1)},\frac{2d-3}{2(2d-1)}\right).
\end{align*}
For $d\ge 2$,  we set  $\mathfrak L_1=[(1/2, 1/2),  \mathfrak A,  \mathfrak G,  \mathfrak G', \mathfrak A']\setminus\{\mathfrak G, \mathfrak G'\}$,   $\mathfrak L_2=[\mathfrak A , (1,1/2),  \mathfrak D]\setminus \{\mathfrak D\}$,    and $\mathfrak L_3=[(1,0),  \mathfrak D,  \mathfrak G,  \mathfrak G', \mathfrak D']\setminus \{ \mathfrak G, \mathfrak D, \mathfrak G', \mathfrak D'\}$ $($Figure \ref{xyxz}$\,)$.
  \end{defn}
  
When $d=2$, $\mathfrak G=\mathfrak G'=\left( 5/6,1/6\right)$ (Figure \ref{xyx}).  When $d\ge 3$,  the line segment $[\mathfrak A, (5/6,1/6)]$ and the line $x-y=2/{(d+1)}$ meet each other at $\mathfrak G$.  See Figure \ref{xyx} and \ref{xyxz}.

\begin{thm}\label{thm-annest}  Let $d\ge 2$ and $\lambda^{-\frac23}\le\mu\le 1/4$. 
If $(1/p,1/q)\in(\mathfrak L_1\cup\mathfrak L_2\cup\mathfrak L_2'\cup\mathfrak L_3)$, \eqref{est-ann} holds.  Moreover,  we have the following estimates: 
\begin{align}
\label{rweak}
&\|\chi_{\lambda,\mu}\Pi_\lambda\chi_{\lambda,\mu} \|_{L^{\frac{2d}{d+1},1}\to L^\infty}\lesssim (\lambda\mu)^{\frac{d-3}{4}}, 
\\ 
\label{weakann}
    &\| \chi_{\lambda,\mu} \Pi_\lambda \chi_{\lambda,\mu} \|_{L^{p,1}\to L^{q,\infty}} \lesssim (\lambda \mu)^{-\frac{1}{d+1}}, \quad ({1}/{p},{1}/{q})=\mathfrak G, \mathfrak G', \quad d\ge 3. 
\end{align}
\end{thm}

Compared with the earlier results,  the range where  \eqref{est-ann} holds is considerably extended. Among others, worth mentioning is  the weak type $(1,\frac{2d}{d-1})$ estimate  which is equivalent to \eqref{rweak} and corresponding to the point $\mathfrak D$ in Figure \ref{local-type}. The estimate makes possible to obtain the sharp estimates  for $\ppq\in \mathcal L_2$.  However,   the optimal bound remains unknown for $(1/p,1/q)\in \operatorname{int} ([\mathfrak A,\mathfrak D,\mathfrak G]\cup [\mathfrak A',\mathfrak D',\mathfrak G'])$. 
  
The proof of the estimate \eqref{est-ann} is more involved. We follow the strategy developed in  \cite{JLR_endpoint}, which makes use of  an explicit integral representation  for the projection operator $\Pi_\lambda$. 

\subsubsection*{Global uniform estimate}
We finally consider the (global) uniform estimates for  $\Pi_\lambda$, that is to say, \eqref{rstest} with $E=\mathbb R^d$ and $\mathrm B$ independent of $\lambda$. Karadzhov \cite{K94} showed  
\begin{align}\label{karad-global}
    \|\Pi_\lambda \|_{2\to \frac{2d}{d-2}}\le C 
    \end{align}
for a constant $C$.   The bound was used to show the sharp $L^p$--Bochner-Riesz summability of the Hermite expansion for $p\ge 2d/(d-2)$ and $p\le 2d/(d+2)$. Besides, the estimate \eqref{karad-global} has applications to the strong unique continuation property for the parabolic operator.  We refer the reader to \cite{e00, ev01,ef03,Fe03,T09,CLWY,CLSY} for related developments. 
   
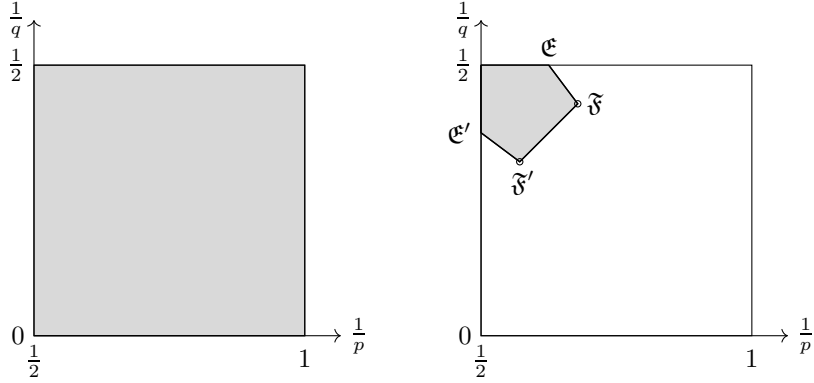
\begin{figure}[t]
\centering
\begin{tikzpicture}[scale=0.6]
\filldraw[fill=gray!30] (0,0)--(6,0)--(6,6)--(0,6);
\draw[<->] (0,7) node[left]{$\frac1q$}--(0,0)--(6.8,0)node[right]{$\frac1p$};
\draw (0,0) rectangle (6,6);
\node[left] at (0,0) {$0$};\node[below] at (0,0) {$\frac12$};
\node[left] at (0,6) {$\frac12$};\node[below] at (6,-0.1) {$1$};
\end{tikzpicture}
\quad \quad 
\begin{tikzpicture}[scale=0.6]
\draw[<->] (0,7) node[left]{$\frac1q$}--(0,0)--(6.8,0)node[right]{$\frac1p$};
\draw (0,0) rectangle (6,6);
\filldraw[fill=gray!30](0,6)--(0,4.5)--(6/7,27/7) --(15/7,36/7)--(1.5,6);
\draw [line width=0.2mm]  (0,6)--(0,4.5)--(6/7,27/7) --(15/7,36/7)--(1.5,6)--(0,6); 
\node[left] at (0,4.5) {$\mathfrak E'$}; \node[above] at (1.5,6) {$ \mathfrak E$}; 
\draw (6/7,27/7) circle [radius=0.07]; \node[below] at (6.5/7,27/7) {$\mathfrak F'$};
\draw (15/7,36/7) circle [radius=0.07]; \node[right] at (15/7,36/7) {$\mathfrak F$}; 
\node[left] at (0,0) {$0$};\node[below] at (0,0) {$\frac12$};
\node[left] at (0,6) {$\frac12$};\node[below] at (6,-0.1) {$1$};
\end{tikzpicture}
\caption{The range of $p,q$ for which \eqref{unif} holds:  $d=2$ (left)  and $d\ge 3$ (right). }
\label{fig:uniform}
\end{figure}

We   obtain  the uniform estimate on an extended range of $p,q$.
 Let 
\begin{align*}
\mathfrak E=\left(\frac{d+2}{2d},\frac 12\right),\; 
\quad 
\mathfrak F=\left(\frac{d^2+2d-4}{2d(d-1)},\frac{d-2}{2(d-1)}\right).
\end{align*}

\begin{thm}\label{glo-est}  Let  $d\ge 3$ and $\mathscr P= [ \mathfrak E, \mathfrak E', \mathfrak F, \mathfrak F', (1/2,1/2)]\setminus\{ \mathfrak F, \mathfrak F'\}.$
Then, 
\Be
\label{unif}
    \|\Pi_\lambda \|_{p\to q}\le C
\Ee
holds for a constant $C$ if  $(1/p,1/q)\in \mathscr P$. Moreover,  $\|\Pi_\lambda \|_{L^{p,1}\to L^{q,\infty}}\le C$ holds if $(1/p,1/q)= \mathfrak F$, $\mathfrak F'$.  When $d=2$,   \eqref{unif} holds for $(1/p,1/q)\in\sq$.
\end{thm}

Related estimates were used to show the strong unique continuation problem for the heat operator \cite{JLR_heat}. 
When $d=2$, the estimate  \eqref{unif}  is easy to show by  duality and the $L^1$--$L^\infty$ estimate.   
In higher dimensions $d\ge 3$,  uniform boundedness of $\|\Pi_\lambda\|_{p\to q}$  remains open for  $(1/p, 1/q)\in   \widetilde {\mathscr P}\setminus \mathscr P$, where $ \widetilde {\mathscr P}:=\{(a,b)\in  \sq:  a-b\le 2/d,   (d-1)/d\le a+b\le (d+1)/d\}$.  Indeed, \eqref{unif} holds  true only  if  $(1/p, 1/q)\in \widetilde {\mathscr P}$ as  can be seen easily by 
 duality and  the lower bounds \eqref{ann1} and \eqref{ann2}  in Section \ref{sub_counter}.  The  current situation seem similar to that of the inhomogeneous Strichartz estimate for the Schr\"odinger equation whose optimal range of boundedness remains open  for $d\ge 3$ (see, for example,  
 \cite{fos, vil}).

\subsection*{Organization} In Section \ref{sec:tt-st}, we formalize a form of $TT^*$ argument for $\Pi_\lambda$, by which we show the uniform estimates for $\Pi_\lambda$. Section \ref{sec:local} is devoted to proving the  local estimates away from $\sqrt \lambda \mathbb S^{d-1}$.  In Section \ref{sec:amu} we prove Theorem \ref{thm-annest}. Finally, we show lower bounds on $\|\chi_{\lambda,\mu}\Pi_\lambda \chi_{\lambda,\mu}\|_{p\to q}$ in Section \ref{sub_counter}.

\subsection*{Notation}  
 For  nonnegative quantities   $A$ and $B$,    $B \lesssim A$ means that there is a constant $C$, depending only on dimensions such that 
$B\le CA $.  Likewise,   $A\sim B$ if and only if   $B \lesssim A$ and $A \lesssim B$. 
By $B=O(A)$ we  means   
$|B|\lesssim A$. Additionally, we denote $A\gg B$ if $A\ge CB$ for a large constant $C>0$.

\section{$\Pi_\lambda$ and $TT^*$ argument}
\label{sec:tt-st}

We make use of an observation in \cite[Section 2.1]{JLR_endpoint}. 
 The  Hermite-Schr\"odinger propagator $e^{-itH}$ is given by 
\Be 
\label{schro-op}
 e^{-itH} f=\sum_{\lambda\in \eval} e^{-it\lambda} \Pi_\lambda f, \quad f\in \mathcal S(\mathbb R^d).
  \Ee
 Clearly,  $e^{it(\lambda - H)}$ is periodic in $t$ with period $\pi$ if $\lambda\in 2\N_0+d$.
If $\lambda$ and  $\lambda'$ are eigenvalues of $H$,   $\lambda-\lambda'\in 2\mathbb Z$, so $\frac1{2\pi}\int_I e^{i\frac t2(\lambda-\lambda')}  dt $ $=\delta(\lambda-\lambda')$ whenever $I$ is an interval of length $2\pi$. It follows from \eqref{schro-op}  that 
  \[
 \Pi_\lambda f = \frac{1}{2\pi}\int_{-\pi/2}^{3\pi/2} e^{i\frac t2(\lambda-H)}f\,dt, \quad \forall f\in \mathcal S(\mathbb R^d).
\]
More details can be found in \cite[Section 2.1]{JLR_endpoint}.

\subsection{Decomposition of $\Pi_\lambda$} 
Let $\eta_\circ$ be an even function in $C_c^\infty((-\pi/2-2^{-7},\pi/2+2^{-7}))$ such that $\sum_{j\in\Z}\eta_\circ(t- j\pi)= 1$ for any $t\in\R$. 
Then, it follows that $\eta_0(t):=\eta_\circ(t+\pi) + \eta_\circ(t) + \eta_\circ(t-\pi) + \eta_\circ(t-2\pi) = 1$ on $[-\pi/2,3\pi/2]$.
So, we can write $\Pi_\lambda f = \frac{1}{2\pi}\int_{-{\pi}/{2}}^{{3\pi}/{2}} \eta_0(t) e^{i\frac{t}{2}(\lambda- H)} f dt$. 
Since $e^{i(t+\pi)(\lambda - H)} = e^{it(\lambda - H)}$, changing  variables we see
\begin{align}
\begin{aligned}\label{proj-op2}
    \Pi_\lambda f = \frac{1}{2\pi}\int_{\R} \big(\eta_\circ(t) + \eta_\circ(t-\pi)\big)e^{i\frac{t}{2}(\lambda- H)} f dt.
\end{aligned}
\end{align}
 The operator $e^{-itH}$ also has an explicit kernel representation based on Mehler's formula 
 (e.g., see  \cite{SoTo} and  \cite[p.11]{Th87}).  Combining the formula and \eqref{proj-op2},
we obtain an integral representation of  $\Pi_\lambda$.

\begin{lem}[{\cite[Lemma 2.1]{JLR_endpoint}}]
\label{kernel}
Let $ \mathfrak a(t) =(2\pi i\:\sin t)^{-\frac d2}{e^{i \pi d/4}}  (\eta_\circ(t) + \eta_\circ(t-\pi)) .$
For $\lambda\in 2\mathbb N_0+d$, set 
\begin{align*}
  \phi_\lambda (x,y,t)=\frac{\lambda t}{2}+\frac{|x|^2+|y|^2}{2}\cot t-  \inp xy \csc t.
\end{align*} Then, for all $f\in \mathcal S(\mathbb R^d)$,  we have 
\Be\label{repPi}
    \Pi_\lambda f =   \frac{1}{2\pi}  \int \mathfrak a(t) \int   e^{i\phi_\lambda(x,y,t)}   f(y)\, dy\,dt.
   \Ee
\end{lem}

The function $\mathfrak a(t)$ has the singularities at $t=0$ and $t=\pi$. So, we make decomposition away from them. 
Let $\psi\in C^\infty_c([2^{-1},2])$ be a nonnegative function  such that $\sum_{j\ge 0} \psi(2^j t) =1$ for $t\in (0, \pi/2+ 2^{-7}]$.  
We set
\[
 \psi_{j}(t) =
    \psi( 2^j t)\eta_\circ(t),
\]
and
\[ \psi_{j}^-(t)=\psi_{j}(- t),\quad \psi_{j}^{\pm\pi}(t)=\psi_{j}(\pm(t-\pi)).  \] 
For a bounded   function $\eta$  and $\lambda\in 2\N_0+d$, we consider the operator 
   \Be \label{proj-pi}\Pi_\lambda[\eta]= \frac{1}{2\pi}\int \eta(t)e^{i\frac t2(\lambda-H)}dt
      .\Ee
Clearly, the definition  makes sense for any real number $\lambda$. By \eqref{proj-pi} and the isometry $
\|e^{-itH} f\|_2=\|f\|_2$ it follows that 
  \Be
  \label{easy-l2} 
  \|\Pi_\lambda[\eta]\|_{2\to 2} \le 2 \|\eta\|_1.
  \Ee

 Since $\sum_j ( \psi_j +\psi^{-} _j)=\eta_\circ$, using \eqref{proj-op2}, we now have
\Be
\label{k-decom}
\textstyle 
\Pi_\lambda=  \sum_{j\ge 0}\Pi_\lambda[\psi_{j}]+    \sum_{\kappa=-,\pm\pi}\sum_{j\ge 0}\Pi_\lambda[\psi_{j}^{\kappa}].
\Ee
The decomposition is clearly valid since the right hand side converges to $\Pi_\lambda$ as a  bounded operator on $L^2$ since \eqref{easy-l2} gives $\| \Pi_\lambda [\psi_j ]\|_{2\to 2}\lesssim 2^{-j}$ and $\| \Pi_\lambda [\psi_j^\kappa ]\|_{2\to 2}\lesssim 2^{-j}$, $\kappa= -, \pm \pi.$   We  recall a symmetric property observed in \cite{JLR_endpoint}. 
 Note $ \phi_\lambda(x,y,-t)=-\phi_\lambda(x,y,t)$ and  $\phi_\lambda(x,y,\pi\pm t)=(\lambda \pi /2) \pm\phi_\lambda(x,-y,t)$.  
 Considering the kernels of the operators (see \cite[(2.8) \& (2,9) in p. 5]{JLR_endpoint}), by a simple change of variables  one can easily show 
\[ \textstyle \| \chi_E \sum_{j\ge 0}\Pi_\lambda[\psi_{j}]  \chi_E\|_{p\to q}= \| \chi_E \sum_{j\ge 0}\Pi_\lambda[\psi_{j}^{\kappa}]  \chi_E\|_{p\to q}, 
    \quad  \kappa=-, \pm \pi \]
    whenever $E$ is a measurable subset such that $E=-E$. 
(See \cite[Section 2]{JLR_endpoint}.) Thus, we get the following which reduces  the desired estimates for $\Pi_\lambda$ to those for  $\sum_j \Pi_\lambda[\psi_j]$. 

\begin{lem}
\label{sym} Suppose $E$ is a measurable subset such that $E=-E$. Then, 
\[ \textstyle \| \chi_E \Pi_\lambda  \chi_E\|_{p\to q}\le 4 \| \chi_E \sum_{j\ge 0}\Pi_\lambda[\psi_{j} ]  \chi_E\|_{p\to q} .\]
The inequality  continues to hold when the spaces $L^p$ and $L^q$ are replaced by the Lorentz spaces $L^{p,r}$ and $L^{q,s}$, respectively.  
\end{lem}

\subsection{$TT^\ast $ argument: Proof of Theorem \ref{glo-est}} 
Lemma \ref{tt-st} below  allows us to deduce off-diagonal estimates from an $L^{r,1}$--$L^{r',\infty}$  bound.    
The following can be regarded as a variant of the usual $TT^\ast$ argument  (see \cite{keel-tao}). 

For $j\ge 0$, we say $\eta\in \mathcal C^j$ if $\eta\in C^\infty_c\big((-2^{1-j},  2^{1-j})\big)$  and $|\eta^{(l)}|\le C_l 2^{j l}$ for some constant $C_l$ and any $l\in \mathbb N_0$. 
In particular, note that $\psi_j\in  \mathcal C^j$. 

\begin{lem}
\label{tt-st}  Let $b> 0$ and $1\le r<r_b:={2(b+1)}/(b+2)$. Set
\[\mathfrak Q(b,r)=[ (1/2, 1/2), (1/2, 1/r_b'), (1/r_b, 1/2), (1/r,1/{r'})]\setminus \{ (1/r,1/{r'})\}.\]  
Let $E\subset \mathbb R^d$ be a measurable set and  suppose 
\begin{equation}
\label{1-infty}
\|\chi_E\Pi_\lambda[\tilde \eta]\chi_E\|_{L^{r,1}\to L^{r',\infty}}\lesssim  \beta^{\delta(r,r')}2^{(-1+(b+1)\delta(r,r'))j},  \quad j\ge0
\end{equation}
holds whenever $\tilde\eta\in \mathcal C^j$ and $\supp  \tilde\eta\subset (2^{-1-j}, 2^{1-j}) $.  
Then, for $j\ge0$ and $(1/p,1/q)\in \mathfrak Q(b,r)$  we have 
\begin{equation} 
\label{bpjb} 
\| \chi_{E} \Pi_\lambda [ \eta] \chi_{E}\|_{p\to q}\lesssim     \beta^{\dpq} 2^{(-1+(b+1)\dpq)j}  
\end{equation}
if $\eta\in \mathcal C^j$ and $\supp  \eta\subset (2^{-1-j}, 2^{1-j}) $. 
\end{lem}

\begin{proof} 
 By \eqref{easy-l2}  we have $\| \chi_{E} \Pi_\lambda [ \tilde \eta ] \chi_{E}\|_{2\to 2}\lesssim    2^{-j}$.  The estimate \eqref{1-infty} and interpolation  give   
\begin{equation}
\label{bpjb2} 
\| \chi_{E} \Pi_\lambda [ \tilde\eta ] \chi_{E}\|_{p\to p'}\lesssim    \beta^{\delta(p, p')}  2^{j(-1+(b+1)(\frac2p-1))}, \quad r< p\le 2 
\end{equation}
whenever $\tilde\eta\in \mathcal C^j$ and $\supp  \tilde\eta\subset (2^{-1-j}, 2^{1-j}) $.
 Thus,  it is sufficient to show  \eqref{bpjb} with $q=2$ and $p=r_b$ because the other estimates follow by duality and interpolation.  
We claim that the estimate 
\begin{equation}
\label{2p}
  \| \Pi_\lambda [ \eta ] \chi_{E}  f\|_{2}\lesssim 2^{-\frac j2} \beta^{\frac1{r_b}-\frac12}
    \|f\|_{r_b}
  \end{equation} 
  holds.  
 The inequality  clearly implies \eqref{bpjb} with $q=2$ and $p=r_b$. 

By \eqref{proj-pi}  we note 
$
\| \Pi_\lambda [ \eta ] \chi_{E} f\|_{2}^2 =\langle   \chi_{E}  \iint \eta(t)\eta(s)e^{i\frac{t-s}2(\lambda-H)}   \chi_{E}  f dsdt, f\rangle.
$ 
Thus, we decompose  
 \Be
 \label{ttsum}
   \| \Pi_\lambda [ \eta ] \chi_{E} f\|_{2}^2 =  \sum_{k\ge j-2}  \left\langle   \chi_{E} P_k  \chi_{E} f, f\right\rangle, 
\Ee
where 
\[   P_k=\iint \psi(2^k|t-s|)\eta(t)\eta(s)e^{i\frac{t-s}2(\lambda-H)}dsdt.
\]
After a simple change of variables we observe 
\Be
\label{bk}
 \chi_{E}P_k  \chi_{E} f= \int\eta(s) \chi_{E}  \Pi_\lambda[\psi(2^k|\cdot|)\eta(\cdot+s)]\chi_{E} f ds.
\Ee
Note $\psi(|\cdot|)=\psi+\psi(-\,\cdot)$.   Since $k\ge j-2$, $\psi(2^k\cdot)\eta(\cdot+s)\in \mathcal C^k$. Thus,  we have  \eqref{bpjb2} with $\tilde \eta=
 \psi(2^k\cdot)\eta(\cdot+s)$ and $j=k$. By the aforementioned symmetric property  of the kernels (\cite[p. 5]{JLR_endpoint}) the same estimate holds for 
 $\Pi_\lambda [\psi(-2^k\cdot)\eta(\cdot+s)]$. Therefore, taking integration in $s$,  we have 
\[ \|\chi_{E}P_k  \chi_{E}\|_{p\to p'} \lesssim \beta^{\delta(p, p')}   2^{-j} 2^{k(-1+(b+1)(\frac2p-1))}\]
 for  $r< p\le 2$. 
 This gives  $ |  \left\langle   \chi_{E} P_k  \chi_{E} f, g\right\rangle|\lesssim    \beta^{(\frac2p-1)}  2^{-j} 2^{k(-1+(b+1)(\frac2p-1))} \|f\|_p\|g\|_p$ by H\"older's inequality.
 If we combine this and \eqref{ttsum}, summation over $k$ yields
\[  \| \Pi_\lambda [ \eta ] \chi_{E} f\|_{2}\lesssim    \beta^{(\frac1p-\frac12)}  2^{j(-1+(b+1)(\frac1p-\frac12))} \|f\|_p\] 
for $r_b<p\le 2$. (Note $r< r_b$.)  Hence, we have \eqref{bpjb} when $q=2$ and $r_b<p\le 2$. Duality gives  \eqref{bpjb} for $p=2$ and $2\le q<r_b'$. Thus, interpolation between those estimates and 
\eqref{bpjb2} gives  \eqref{bpjb} if  $\ppq$ is contained in $\mathfrak Q(b,r)$ but  not on the line segments 
$[(1/2,1/r_b'), (1/r,1/{r'})]$, $[(1/r_b,1/2), (1/r,1/{r'})]$.  

However, using  the estimates above, we can obtain \eqref{2p}. Indeed, using \eqref{bk} and  \eqref{bpjb} which now holds for $\ppq$ contained in the interior of $\mathfrak Q(b,r)$, by 
H\"older's inequality   we have
\Be \label{bibi} 
|  \left\langle   \chi_{E} P_k  \chi_{E} f, g\right\rangle|\lesssim    \beta^{\dpq}  2^{-j} 2^{k(-1+(b+1)\dpq)} \|f\|_p\|g\|_{q'}
\Ee
for  $(1/p, 1/q)\in \operatorname{int}\mathfrak Q(b,r)$. This allows us to apply the bilinear interpolation argument 
(e.g.,  Keel and Tao \cite{keel-tao}). Therefore, we obtain
\[ \textstyle \sum_{k\ge j-2}  | \left\langle   \chi_{E} P_k  \chi_{E} f, g\right\rangle|\lesssim    \beta^{\dpq}  2^{-j} \|f\|_p\|g\|_{q'}\] 
provided that $(1/p, 1/q)\in \operatorname{int}\mathfrak Q(b,r)$ and $(b+1)(1/p-1/q)=1$. In particular, taking $q'=p$,  by  \eqref{ttsum}  we obtain \eqref{2p}.
 \end{proof}

The following lemma  is useful for obtaining some endpoint estimates. 

\begin{lem}\label{s-trick}  Let $1\le p_0, p_1,  q_0, q_1\le \infty$ and $\epsilon_0, \epsilon_1>0$.  
Let $T_j$, $j\in \mathbb Z $, be sublinear operators satisfying $\|T_j\|_{p_k\to q_k}\le B_k 2^{ j (-1)^k \epsilon_k }$ for $k=0,1$.  Let $\theta= {\epsilon_0}/(\epsilon_0+\epsilon_1)$, 
$1/p_\ast=\theta/{p_1}+ (1-\theta)/{p_0},$ and $1/q_\ast=\theta/{q_1}+(1-\theta)/{q_0}. $  
Then, the following hold\,$:$
\vspace{3pt}

\begin{enumerate} 
[leftmargin=1.2cm, labelsep=0.15 cm, topsep=0pt]
\item[$(a)$] If $p_0=p_1=p$ and $ q_0\neq q_1$, then   $\|\sum_j T_j f  \|_{{q_\ast,\infty}}\lesssim B_0^{1-\theta}B_1^\theta \|f\|_{p}$, 
\item[$(b)$] If $q_0=q_1=q$ and $ p_0\neq p_1$, then $\|\sum_j T_j f  \|_{{q}}\lesssim B_0^{1-\theta}B_1^\theta \|f\|_{p_\ast,1}$,
\item[$(c)$] If $p_0\neq p_1$ and $ q_0\neq q_1$, then  $\|\sum_j T_j f  \|_{{q_\ast,\infty}}\lesssim B_0^{1-\theta}B_1^\theta \|f\|_{p_\ast, 1}$.
\end{enumerate}
\end{lem} 

The third assertion $(c)$ is known as \emph{`Bourgain's summation trick'} (see \cite[Section 6.2]{Car99} for a formulation in abstract setting).  
The first $(a)$ and the second $(b)$  give   better estimates than the restricted weak type estimate. As far as the authors are aware, this observation first appeared in \cite{bak} (see also  \cite[Lemma 2.3]{LS}).

\begin{rem}  Thanks to Lemma \ref{s-trick} and the estimate  \eqref{bibi}, which is equivalent to $\|\chi_{E} P_k  \chi_{E} \|_{p\to q} \lesssim \beta^{\dpq}  2^{-j} 2^{k(-1+(b+1)\delta(p,q))}$,   an elementary approach to the estimate \eqref{2p}   
 is possible.  By  $(c)$ in Lemma \ref{s-trick}   we have   $\|\sum_{k}\chi_{E} P_k  \chi_{E} f\|_{q,\infty}\lesssim  \beta^{\delta(p,q)}  2^{-j}  \|f\|_{p,1}$ for $\ppq\in \operatorname{int}\mathfrak Q(b,r)$ satisfying $(b+1)\,\delta(p,q)=1$.   Interpolation gives, in particular,   \eqref{bpjb} with  $p=r_b$ and $q=r_b'$, which  is equivalent to \eqref{2p}.
\end{rem}

By using Lemma \ref{tt-st}, we can prove Theorem \ref{glo-est}.

\begin{proof}[Proof of Theorem \ref{glo-est}] 
To show \eqref{unif} for $(1/p,1/q)\in \mathscr P$,  by  interpolation  it suffices to show  the restricted weak type $(p,q)$  estimate for $\ppq=\mathfrak F, \mathfrak F'$. 
By duality  we need only  to show 
\Be\label{restricted-weak}
\|\Pi_\lambda f \|_{q_\circ,\infty}\le C\|f\|_{p_\circ,1},   
\Ee
where $(1/p_\circ, 1/q_\circ)=\mathfrak F'.$ Indeed, once  we have \eqref{restricted-weak}, duality and interpolation give  \eqref{unif} for 
$(\rp,\rq)$ $\in (\mathfrak F,\mathfrak F')$. Note
$\Pi_\lambda^* \Pi_\lambda =\Pi_\lambda$ and $(\tfrac{d+2}{2d}, \tfrac{d-2}{2d})\in (\mathfrak F,\mathfrak F')$. So, we get \eqref{karad-global} since 
$\|\Pi_\lambda\|_{2\to p'}^2=\|\Pi_\lambda\|_{p\to p'}$.  Besides, duality gives  $\|\Pi_\lambda \|_{\frac{2d}{d+2} \to 2}\le C$. 
Interpolation  between those estimates  and $\|\Pi_\lambda \|_{2\to 2}\le 1$ gives \eqref{unif} for $\ppq\in \mathscr P$ (see Figure \ref{fig:uniform}).

By Lemma \ref{sym}  it is sufficient for \eqref{restricted-weak}  to show 
\Be\label{restricted-weak_p}
\textstyle \|\sum_{j\ge 0} \Pi_\lambda [\psi_j]f \|_{q_\circ,\infty}\le C\|f\|_{p_\circ,1}.
\Ee
By Lemma \ref{kernel}  we have \eqref{1-infty} with $E=\mathbb R^d$, $\beta=1$, $r=1$, and $b=(d-2)/{2}$,  provided that 
$\tilde\eta\in \mathcal C^j$ and $\supp  \tilde\eta\subset (2^{-1-j}, 2^{1-j}) $. 
Using this and  Lemma \ref{tt-st}, we have  
\Be
\label{eq:strong}
\|   \Pi_\lambda [\psi_j  ] \|_{p\to q}\lesssim   2^{j(\frac{d}{2}(\frac1p-\frac1q)-1)},  \quad j\ge 0,
\Ee
for $(\rp, \rq)$ contained in $\mathfrak  Q((d-2)/2,1)$ which is a quadrangle with vertices $(1/2, 1/2)$,  $(1/2, (d-2)/(2d))$, 
$( (d+2)/(2d), 1/2)$, and  $(1,0)$. To show \eqref{restricted-weak_p}, we note that   $1/p_\circ-1/q_\circ=2/d$ and make use of   the summation trick ($(c)$ in Lemma \ref{s-trick}). 
Using \eqref{eq:strong}, we get restricted weak type $(p,q)$ estimate for $\sum_{j\ge 0} \Pi_\lambda[\psi_j]$  if 
$(\rp, \rq)\in\mathfrak  Q((d-2)/2,1)$ and $\rp-\rq=2/d$.  We only need to observe that  $\mathfrak  Q((d-2)/2,1)\cap \{ (x,y)\in \sq: x-y=2/d\}=[\mathfrak F, \mathfrak F']$. 
\end{proof}

\subsection{$L^2$--$L^\infty$ estimate}
We now consider $L^2$--$L^\infty$ estimate for  $\Pi_\lambda[\eta]$, which plays a significant role in what follows.  
For a given operator $T$, by $T(x,y)$ we denote the kernel of $T$. A simple duality argument shows  
\Be
\label{2-infty}
\|T\|_{2\to \infty}=\|T\|_{L_x^\infty(L_y^2)}. 
\Ee
We also observe that  
\Be
\label{spectral-proj}
\textstyle \Pi_\lambda [\eta]   f= \frac1{2\pi} \sum_{\lambda'} \widehat {\eta} ( 2^{-1}(\lambda'-\lambda))\Pi_{\lambda'} f
\Ee
for  $\eta\in C_c^\infty$, which follows  from \eqref{schro-op} and \eqref{proj-pi}. 

\begin{lem}
\label{PI-est}  Let $ \lambda^{-\frac 23}\lesssim \mu \le 1/4$ and $ 2^{-j}\gtrsim  (\lambda\mu)^{-1} $. 
If $\eta\in \mathcal C^j$, then there is a constant $C$, independent of   $\lambda$ and $\mu$, such that 
\begin{align}
\label{201}
\|\chi_{\lambda,\mu}\Pi_\lambda[\eta]\|_{2\to\infty}&\lesssim 2^{-\frac j2}(\lambda\mu)^{\frac{d-2}{4}},
\\
\label{202}
\|\Pi_\lambda[\eta]\|_{2\to\infty}&\lesssim 2^{-\frac j2}\lambda^{\frac{d-2}{4}}.
\end{align}
\end{lem}

\begin{proof}  We first show \eqref{201}. 
The proof follows an argument similar to that of Lemma 2.9 in \cite{JLR_endpoint}.

By orthogonality,  $\|\Pi_\lambda[\eta](x,\cdot)\|_2^2 = \pi^{-2} \sum_{\lambda'} |\widehat \eta(2^{-1}(\lambda'-\lambda))|^2 |\Pi_{\lambda'}(x,x)|$. Also note that $|\widehat \eta(\tau)|\lesssim 2^{-j}(1+ 2^{-j} |\tau|)^{-N}$ for any $N$. Thus, by \eqref{spectral-proj} and \eqref{2-infty},  we  have 
\begin{align*}
\|\chi_{\lambda,\mu}\Pi_\lambda[\eta]\|_{2\to\infty}^2 
\lesssim  \sup_{x\in A_{\lambda, \mu}} \sum_{\lambda'} 2^{-2j}\big(1+2^{-j}|\lambda-\lambda'|\big)^{-N}\Pi_{\lambda'}(x,x).
\end{align*}
Since $\|T\|_{1\to \infty}=\|T\|_{L^\infty_{x,y}}$ for an operator $T$, this reduces the proof of \eqref{201} to showing  
\Be
\label{1010}
\tx  \sum_{\lambda'} 2^{-2j}\big(1+2^{-j}|\lambda-\lambda'|\big)^{-N}\|\chi_{\lambda,\mu}\Pi_{\lambda'}\chi_{\lambda,\mu}\|_{1\to\infty} \lesssim 2^{-j}(\lambda\mu)^{\frac{d-2}{2}}
\Ee
for $ 2^{-j}\gtrsim  (\lambda\mu)^{-1} $.   Using 
$\Pi_{\lambda'}^2=\Pi_{\lambda'}$ and duality, we  note $\|\chi_{\lambda,\mu}\Pi_{\lambda'}\chi_{\lambda,\mu}\|_{1\to\infty} \le \|\chi_{\lambda,\mu}\Pi_{\lambda'}\|_{2\to\infty}^2$. 
Thus, the estimate \eqref{1010}  follows  from a  stronger estimate 
\Be 
\label{pi-sum} 
  \sum_{\lambda'} \big(1+2^{-j}|\lambda-\lambda'|\big)^{-N} \|\chi_{\lambda,\mu}^e\Pi_{\lambda'}\|_{2\to\infty}^2 \le C2^{j}(\lambda\mu)^{\frac{d-2}{2}}, 
\Ee
where $\chi_{\lambda,\mu}^e = \chi_{A_{\lambda,\mu}^e}$ and $A_{\lambda,\mu}^e= \{x:|x|\ge\lambda^{1/2}(1-\mu)\}$. 

 To handle the sum above, we use  the estimates 
\begin{align}\label{kl2infty}
    \|\chi_{\lambda,\mu}\Pi_\lambda\|_{2\to \infty}\lesssim (\lambda\mu)^{\frac{d-2}{4}},
\end{align}
$\|\chi_{A_{\lambda}^\circ}\Pi_\lambda\|_{2\to\infty}\lesssim \lambda^{\frac{d-2}{12}}$, and $\|\chi_{A_{\lambda,\mu}^-}\Pi_\lambda\|_{2\to\infty}\lesssim \lambda^{\frac{d-2}{12}}(\lambda\mu^{\frac32})^{-N}$, where $A_{\lambda}^\circ := \{\sqrt\lambda x: |1-|x||\le\lambda^{-2/3}\}$ and $A_{\lambda,\mu}^-:=\{\sqrt \lambda x: \mu\le|x|-1<2\mu\}$.  The first estimate \eqref{kl2infty} follows from \eqref{kt-nu}.
For the second and third estimates, the reader may find their proofs in \cite{JLR_endpoint} or \cite{T05}.
 Since $ \lambda^{-\frac 23}\lesssim \mu$, summation  over annuli gives
\Be
\label{outside} \|\chi_{\lambda,\mu}^e\Pi_\lambda\|_{2\to\infty}\lesssim (\lambda\mu)^{\frac{d-2}{4}}.
\Ee

We define $\ell(\rho)=\ell(\rho, \lambda, \lambda')$ by setting 
$\ell(\rho)= (\lambda/\lambda')^{\frac23}\rho$ if $\lambda \ge \lambda'$ and 
     $\ell(\rho)= (\lambda'-\lambda)/\lambda+\rho$ otherwise.  Since  $\lambda^{-\frac23}\lesssim\mu$,  we have $\ell(\mu) \gtrsim(\lambda')^{-\frac23}$. Note that  $(\lambda')^\frac12(1-\ell(\mu))\le \lambda^\frac12(1-\mu)$. Thus,  it follows  
 that  $  \|\chi_{\lambda,\mu}^e \Pi_{\lambda'} \|_{2\to\infty}\le \|\chi_{\lambda', \ell(\mu)}^e \Pi_{\lambda'}\|_{2\to\infty}$. 
 By \eqref{outside} the left hand side of \eqref{pi-sum} is bounded above by 
 \begin{align*}
 \textstyle C  \sum_{\lambda'} \big(1+2^{-j}|\lambda-\lambda'|\big)^{-N} (\lambda'\ell(\mu))^{\frac{d-2}{2}}.
 \end{align*}
 To prove \eqref{pi-sum} it is sufficient to show the above sum is bounded by $C2^j(\lambda\mu)^{(d-2)/2}$.  
 Indeed, considering separately the cases $\lambda\ge \lambda'$ and $\lambda<\lambda'$,  we only have to show
 $ \sum_{\lambda'\le \lambda} \big(1+2^{-j}(\lambda-\lambda')\big)^{-N} (\lambda'/\lambda)^{\frac{d-2}{6}}
 \lesssim  2^j $ and 
  \begin{align*}
   \textstyle  \sum_{\lambda'>\lambda} \big(1+2^{-j}(\lambda'-\lambda)\big)^{-N}& (\lambda'/\lambda)^\frac{d-2}2 \big((\lambda'-\lambda)/(\lambda\mu)+1\big)^{\frac{d-2}{2}}
      \lesssim  2^j.
 \end{align*}
Both follow from a simple computation. Particularly, we use $ 2^{-j}\gtrsim  (\lambda\mu)^{-1}$ for the second inequality. 

One can easily show the estimate \eqref{202}    in the same manner using  $\|\Pi_\lambda\|_{2\to\infty}\lesssim \lambda^{\frac{d-2}{4}}$. So, we omit the detail. 
\end{proof}

\subsubsection*{The marginal case  $2^{j}\ge \lambda\mu$}  Making use of the previous estimates, we obtain  estimates for $\sum_{2^{j}\ge \lambda\mu} \chi_{\lambda,\mu}\Pi_\lambda[\psi_j] \chi_{\lambda,\mu}$, whose contribution turns out to be less significant. 

\begin{lem}\label{imply1}
Let $\mu\in \mathbb D^{-1}:=\{2^{-k}:\,-k\in\mathbb N\}$ and $2^{j_\circ-1}\le \lambda\mu< 2^{j_\circ}$. Suppose $\eta\in \mathcal C^{j_\circ}$. Then, for $(1/p,1/q)\in \sq$, we have
\begin{align}
\label{imp22}
\| \chi_{\lambda,\mu}\Pi_\lambda[\eta]\chi_{\lambda,\mu}\|_{p\to q}\lesssim \lambda^{\beta(p,q)}\mu^{\gamma(p,q)}.
\end{align}
\end{lem}

\begin{proof}
Note $(\lambda\mu)^{-1+\frac d2\delta(p,q)}\le \lambda^{\beta(p,q)}\mu^{\gamma(p,q)}$. So, \eqref{imp22} follows once we have
\begin{align}\label{imp2}
   \| \chi_{\lambda,\mu}\Pi_\lambda[\eta]\chi_{\lambda,\mu}\|_{p\to q} \lesssim (\lambda\mu)^{-1+\frac d2\delta(p,q)}.
\end{align}
In view of interpolation and duality  it is sufficient to show  \eqref{imp2} for $(p,q)=(2,2)$, $(2,\infty)$, $(1,\infty)$.  
The estimate   \eqref{imp2}  for $(p,q)=(2,2)$  is clear from \eqref{easy-l2} since $\|\eta\|_1\lesssim 2^{-j_\circ}$.  Since $ 2^{j_\circ}\sim \lambda\mu$,   \eqref{imp2}  for $(p,q)=(2,\infty)$ follows from \eqref{201}.

Note that  $\widehat\eta(\tau) = O(2^{-j_\circ}(1+2^{-j_\circ}|\tau|)^{-N})$ for any $N\in\N$. Using \eqref{spectral-proj}, we obtain  
\[
\textstyle 
\|\chi_{\lambda,\mu}\Pi_\lambda[\eta]\chi_{\lambda,\mu}\|_{1\to \infty}\lesssim  \sum_{\lambda'}2^{-j_\circ}(1+2^{-j_\circ}|\lambda'-\lambda|)^{-N}\|\chi_{\lambda,\mu}\Pi_{\lambda'}\chi_{\lambda,\mu}\|_{1\to \infty}.
\] 
Since $ 2^{j_\circ}\sim \lambda\mu$,  \eqref{imp2} for $(p,q)=(1,\infty)$ follows from \eqref{1010}.  
\end{proof}

\subsection{Estimates for the kernel of $\Pi_\lambda$}
In this subsection we are concerned with estimates for  the kernel of $\Pi_\lambda[\psi_j]$. 
To do so, it is more convenient to consider  a rescaled operator.  For $\eta\in C_c^\infty (\mathbb R)$,  by  $\fP_\lambda[\eta]$  we denote the operator whose kernel is given by  
\[\fP_\lambda[\eta](x,y)= \Pi_\lambda[\eta ](\sqrt \lambda x, \sqrt \lambda y).\]
 As before, by Mehler's formula (cf. \eqref{repPi})  and scaling  one can see 
\Be 
\label{eq:scaled-op}
 \fP_\lambda[\eta](x,y)=\frac1{2\pi}\int \eta(t)\mathfrak a(t)e^{i\lambda \phi_1(x,y,t)} dt. 
\Ee

To obtain estimates for $\fP_\lambda[\psi_j]$ we  examine the phase function of the oscillatory integral. 
A calculation shows
\begin{align}
\label{ph-d}
   \partial_s \phi_1(x,y,s) 
           =-\frac{\mathcal{Q}(x,y,\cos s)}{2\sin^2s}, 
          \end{align}
where  
\begin{align}
\label{q-def} 
  \mathcal{Q}(x,y, \tau)
&:=(\tau- \inp xy)^2-\mathcal D(x,y),
\\
 \cD(x,y)&:= 1+\inp xy^2 -|x|^2-|y|^2.
 \end{align}
 
The stationary point of $\phi_1(x,y,\cdot)$ is  given by the zeros of $\cQ(x,y,\cos \cdot)$. 
 $\cD(x,y)$, which is the discriminant of the quadratic equation $\cQ(x,y,\tau)=0$,  regulates  the nature of stationary point of the phase function $\phi_1(x,y,\cdot)$.   
 In fact, we can obtain bounds on $\fP_\lambda[\psi_j](x,y)$ in terms of $|\cD(x,y)|$, which are to be used later.

\begin{lem}
\label{oscillatory} Let $\mu\in \mathbb D^{-1}$ and let  $x,y\in \R^d$ satisfy that $1-C_1\mu\le |x|, |y|\le 1-c_1\mu$ for some $0<c_1<C_1$. If $1+ \inp xy\ge 10^{-2}$, then for any $N>0$ there exists a constant $\mathrm B = \mathrm B(c_1,C_1,N,d)$ such that the following hold for $j\ge 0$:
\begin{enumerate}[leftmargin=1.3cm, labelsep=0.3 cm, topsep=0pt]
    \item [$(a)$] If $-\cD(x,y)\gtrsim \mu^2$, then
    \begin{align*}
        |\mathfrak P_\lambda[\psi_j](x,y)|\le \mathrm B \begin{cases}
           2^{\frac{d-2}{2}j}(\lambda 2^j |\cD(x,y)|+1)^{-N}, &  \  2^{-j}\lesssim |\cD(x,y)|^{\frac14}, \\[2pt]
           2^{\frac{d-2}{2}j}(\lambda 2^{-3j}+1)^{-N}, &  \ 2^{-j}\gg |\cD(x,y)|^{\frac14}.
        \end{cases}
    \end{align*}
    \item[$(b)$] If $|\cD(x,y)|\ll \mu^2$, then
    \begin{align*}
        |\mathfrak P_\lambda[\psi_j](x,y)|\le \mathrm B \begin{cases}
           2^{\frac{d-2}{2}j}(\lambda 2^j \mu^2+1)^{-N}, &  \  2^{-j}\ll \mu^{\frac12}, \\[2pt]
           2^{\frac{d-2}{2}j}(\lambda 2^{-3j}+1)^{-N}, & \  2^{-j}\gg \mu^{\frac12}.
        \end{cases}
    \end{align*}
    \item[$(c)$] If $\cD(x,y)\sim \mu^2$, then
    \begin{align*}
        |\mathfrak P_\lambda[\psi_j](x,y)|\le \mathrm B \begin{cases}
           2^{\frac{d-1}{2}j}\lambda^{-\frac12}\mu^{-\frac12}, & \  2^{-j}\lesssim \mu^{\frac12}, \\[2pt]
           2^{\frac{d-2}{2}j}(\lambda 2^{-3j}+1)^{-N}, & \  2^{-j}\gg \mu^{\frac12}.
        \end{cases}
    \end{align*}
\end{enumerate}
If $1+\inp xy < 10^{-2}$, then $|\fP_\lambda[\psi_j](x,y)|\lesssim 2^{\frac{d-2}{2}j}(\lambda 2^j)^{-N}$ for $j\ge 0$.
\end{lem}

\begin{proof}
We first consider the case $1+\inp xy \ge 10^{-2}$ and show $(a)$--$(c)$.   The estimates  $(a)$ and $(b)$ can be shown in a similar way. 
We begin with observing that $0<1-\inp xy\sim 1-\inp xy^2$ since $1+\inp xy \ge 10^{-2}$. So, we have 
\Be \label{relation}
  1-\inp xy\sim  2-|x|^2-|y|^2-\cD(x,y)\sim \mu 
\Ee
if $-\cD(x,y)\gtrsim \mu^2$ or $|\cD(x,y)| \ll \mu^2$.

Note that $1-\cos s\sim 2^{-2j}$ if $s\in \supp \psi_j$. When 
$|\cD(x,y)|\gtrsim \mu^2$, recalling \eqref{q-def} and using \eqref{relation}  we see  $ \mathcal Q(x,y, \cos s) \gtrsim 2^{-4j} $ for
$2^{-j}\gg |\mathcal D(x,y)|^{1/4}$ and  $ s\in\supp{\psi_j} $.  Thus, from \eqref{ph-d} we have 
\begin{align*}
    |\partial_s\phi_1(x,y,s)|
    \gtrsim 
    \begin{cases}
       2^{2j}|\cD(x,y)|,& 2^{-j}\lesssim |\cD(x,y)|^{\frac14}, 
       \\
       2^{-2j}, & 2^{-j}\gg|\cD(x,y)|^{\frac14}
    \end{cases}
\end{align*}
for $ s\in\supp{\psi_j} $ if  $-\cD(x,y)\gtrsim \mu^2$. 
We also note that  $|(d/ds)^n((\sin s)^{-2})|\lesssim 2^{(2+n)j}$ and $|\partial_s^n(\cos s - \inp xy)|\lesssim 2^{nj}\max(\mu,2^{-2j})$ for any $n\in\N_0$. Being combined with \eqref{ph-d}, these bounds yield  
\begin{align*}
    |\partial_s^n\phi_1(x,y,s)|\lesssim \begin{cases}
       2^{(1+n)j}|\cD(x,y)|,& 2^{-j}\lesssim |\cD(x,y)|^{\frac14} \\
       2^{(n-3)j}, & 2^{-j}\gg|\cD(x,y)|^{\frac14}
    \end{cases} ,\quad s\in\supp{\psi_j}
\end{align*}
if  $\cD(x,y)\lesssim -\mu^2$. 
Also, we have $|(d/ds)^n( 2^{-dj/2}\mathfrak a\psi_j)|\lesssim 2^{nj}$, $n\in\N_0$. Thus, integration by parts for the integral $\mathfrak P_\lambda[\psi_j](x,y)$ (recall \eqref{eq:scaled-op})  gives the estimate in $(a)$ (see, for example, \cite[Lemma 2.5]{JLR_endpoint}). 

We can show $(b)$ in the same manner as above.  Since $1-\cos s\sim 2^{-2j}$ and $|\cD(x,y)|\ll \mu^2$,  we have  $ \mathcal Q(x,y, \cos s) \sim \mu^2 $ if $2^{-j}\ll \mu^{\frac12}$ 
and $ \mathcal Q(x,y, \cos s) \sim 2^{-4j} $ if $2^{-j}\gg \mu^{\frac12}. $  Thus,  using \eqref{ph-d} we have \begin{align*}
    |\partial_s\phi_1(x,y,s)|\gtrsim \begin{cases}
       2^{2j}\mu^2,& 2^{-j}\ll \mu^{\frac12} \\
       2^{-2j}, & 2^{-j}\gg\mu^{\frac12} 
    \end{cases},\quad s\in\supp{\psi_j}.
\end{align*}
Similarly as in the proof of $(a)$,  we get 
\begin{align*}
    |\partial_s^n\phi_1(x,y,s)|\lesssim \begin{cases}
       2^{(1+n)j}\mu^2,& 2^{-j}\ll \mu^{\frac12} \\
       2^{(n-3)j}, & 2^{-j}\gg\mu^{\frac12}
    \end{cases}, \quad  s\in\supp{\psi_j}
\end{align*}
for any $ n\in\N$.  Combining those and $|(d/ds)^n( 2^{-dj/2}\mathfrak a\psi_j)|\lesssim 2^{nj}$, by routine integration by parts we obtain the desired  estimate in $(b)$ (\cite[Lemma 2.5]{JLR_endpoint}).

We now prove $(c)$.  From \eqref{relation} we note  $1-\inp xy\lesssim \mu$ since $\cD(x,y)\sim \mu^2$.  Using this, we have $\mathcal Q(x,y, \cos s) \sim 2^{-4j} $ if $2^{-j}\gg \mu^{\frac12}$.
Thus, the estimate for the case $2^{-j}\gg\mu^{1/2}$ can be obtained in the same manner as in the proof of $(b)$. 
Therefore, we only need to show $(c)$ assuming 
 that 
 \[ 2^{-j}\lesssim \mu^{\frac12}.\]
In this case  $\supp \psi_j$ may contain at least one of stationary points of $\phi_1(x,y,\cdot)$.
The equation $\mathcal Q(x,y,\cdot) = 0$ has two roots $\tau_\pm=\inp xy\pm\sqrt{\cD(x,y)}$.
Since $\cQ(x,y,1)=|x-y|^2$ and $\cQ(x,y,-1)=|x+y|^2$, $\tau_\pm\in [1, -1]$. 
Thus, the function $\partial_s\phi_1(x,y,\cdot)$ has two  zeros $s_-$, $s_+$ on $[0,\pi]$ such that
\[
\cos s_\pm = \inp xy \pm \sqrt{\cD(x,y)}.
\]
Since $\cos s_+ - \cos s_- = 2\sqrt{\cD(x,y)}\sim \mu$, we have $s_- - s_+ \gtrsim \mu^{1/2}$.
We dyadically decompose the integral $\mathfrak P_\lambda[\psi_j](x,y)$ (\eqref{eq:scaled-op})  away from $s_\pm$.
To do this, we let
\[ \eta_k^\kappa(s) = \psi(2^k|s-s_\kappa |), \quad \kappa = \pm, \]
 and $k_0$ be an integer such that $\mu^{1/2}/(2C)\le 2^{-k_0}<\mu^{1/2}/C$ for a large constant $C>0$.
We also set 
\[
\tx \eta^*(s) = 1 - \sum_{k>k_0}\eta_k^+(s) - \sum_{k>k_0}\eta_k^-(s)
\]
so that $\tx \eta^\ast+ \sum_{k>k_0}\eta_k^++ \sum_{k>k_0}\eta_k^-=1$. 

We first estimate the sum $\sum_{k>k_0} \mathfrak P_\lambda[\psi_j \eta_k^+](x,y)$. 
From  \eqref{ph-d},  note that
\[
 \partial_s\phi_1(x,y,s)= -\frac{(\cos s - \cos s_-)}{2\sin^2 s}\int_{s_+}^s \sin u \,du.
\]
Since $\cos s_+ -\cos s_-\sim \mu$, we have  $|\cos s - \cos s_-|\sim \mu$ on  $\supp{(\psi_j\eta_k^+)}$. We also note $\big|\int_{s_+}^s \sin u du\big|\sim 2^{-j}2^{-k}$ for $s\in\supp{(\psi_j\eta_k^+)}$.
Hence,  $|\partial_s\phi_1(x,y,s)|\gtrsim 2^j2^{-k}\mu$ for $s\in \supp{(\psi_j\eta_k^+)}$.
By the van der Corput lemma (for example, \cite[pp.\,332--334]{St93}) we get  $| \mathfrak P_\lambda[\psi_j \eta_k^+](x,y) |\lesssim 2^{\frac{d-2}{2}j}\lambda^{-1}\mu^{-1}2^k$.
Combining this and a trivial bound  $| \mathfrak P_\lambda[\psi_j \eta_k^+](x,y) |\lesssim 2^{\frac d2j}2^{-k}$,  we obtain 
\[
\textstyle \big|\sum_{k>k_0} \mathfrak P_\lambda[\psi_j \eta_k^+](x,y) \big|\lesssim \sum _{k>k_0} 2^{\frac d2j}\min(2^k\lambda^{-1}2^{-j}\mu^{-1},\,2^{-k})\lesssim 2^{\frac{d-1}{2}j}\lambda^{-\frac12}\mu^{-\frac12}.
\]
In the same manner, one can show $|\sum_{k>k_0} \mathfrak P_\lambda[\psi_j \eta_k^-](x,y)|\lesssim 2^{\frac{d-1}{2}j}\lambda^{-\frac12}\mu^{-\frac12}$.

To complete the proof of $(c)$,  it remains to show  $|\mathfrak P_\lambda[\psi_j \eta^\ast](x,y)|\lesssim 2^{\frac{d-1}{2}j}\lambda^{-\frac12}\mu^{-\frac12}$. Note that $\cQ(x,y, \cos s)\sim \mu^2$ for $s\in \supp(\psi_j\eta^*)$. 
This gives  $|\partial_s\phi_1(x,y,s)|\gtrsim 2^{2j}\mu^2$. Thus, by the van der Corput lemma we get  $|\mathfrak P_\lambda[\psi_j \eta^\ast](x,y)|
\lesssim 2^{\frac{d-4}{2}j}\lambda^{-1} \mu^{-2}$. Combining this and a trivial bound $|\mathfrak P_\lambda[\psi_j \eta^\ast](x,y)|
\lesssim 2^{\frac{d}{2}j} \mu^{1/2}$, we obtain the desired estimate  since $2^{-j}\lesssim \mu^{1/2}$.

To prove the last assertion, we observe that  $  |\cos s - \inp xy|\ge 10^{-1}$ for $s\in\supp{\psi_j}$ and  $\cD(x,y)= (1+\inp xy)^2 - |x+y|^2 \le 10^{-4}$ if $1+\inp xy < 10^{-2}$.
Thus, $\mathcal Q(x,y,\cos s)\sim 1$ for $s\in\supp{\psi_j}$. Combining this and \eqref{ph-d}, we have $|\partial_s\phi_1|\gtrsim 2^{2j}$.
Therefore,  repeated integration by parts, as before,  gives the desired estimate.
\end{proof}

\begin{cor} Let  $E\subset \R^d$. Suppose that $|\cD(x,y)|\ge c_0$ for a constant $c_0>0$ whenever $x,y\in E$. Then, there is a constant 
$C=C(c_0)$ such that  
\Be
\label{che}
\|   \che \Pi_\lambda [\psi_j ] \che \|_{1\to \infty}
 \le  C  2^{\frac{d-1}2j}\lambda^{-\frac12}, \quad j\ge 0.
\Ee
\end{cor}
\begin{proof}
By rescaling $(x,y)\to (\sqrt \lambda x,  \sqrt \lambda y)$, \eqref{che} is equivalent to  
\begin{equation}\label{e:simple}
|\mathfrak P_\lambda[\psi_j](x,y)|\le C 2^{\frac{d-1}2j}\lambda^{-\frac12}
\end{equation}
for $x,y\in E$ and $j\ge 0$. Since $\cD(x,x)=(1-|x|^2)^2$, we have $|x|\le 1-c$ or $|x|\ge1+c$ for a small constant $c=c(c_0)>0$ if $x\in E$. 

Let $x,y\in E.$ Then,  by symmetry we need only to consider the cases $|x|,|y|\le 1-c$; $|x|\le 1-c, 1+c\le |y|$; and  
$1+c\le |x|, |y|$. For the first case,  by choosing $C_1=2$ and $c_1=2c$ in Lemma \ref{oscillatory}, we may assume that $1-C_1\mu\le |x|,|y|\le 1-c_1\mu$ with $\mu=1/2$. Thus, 
\eqref{e:simple} follows from $(a)$ or $(c)$ in Lemma \ref{oscillatory}. When  $|x|\le 1-c$ and $1+c\le |y|$, we have $\cD(x,y)\le (1-|x|^2)(1-|y|^2)\le -c^2$. 
By \eqref{ph-d} and \eqref{q-def}  this gives  $|\partial_s \cP(x,y,s)|\gtrsim 2^{2j}$ for $s\in \supp \psi_j$. Using  van der Corput's lemma, we get 
 \[ |\mathfrak P_\lambda[\psi_j](x,y)| \lesssim 2^{\frac{d}2j}\min(\lambda^{-1}2^{-2j},\,2^{-j})\lesssim 2^{\frac{d-1}2j}\lambda^{-\frac12}.\]

For the third case $1+c\le |x|, |y|$, we may assume  $\cD(x,y)\ge c_0$ since the estimate \eqref{e:simple} follows by the same argument  as above if $\cD(x,y)\le -c_0$. 
Thus, we have $1+c_0 < \inp xy^2$. If $\inp xy> (1+c_0)^{1/2}$,  the two distinct roots $r_1< r_2$ of the equation $\mathcal Q(x,y,\tau)=0$ are bigger than or equal to $1$ because $\mathcal Q(x,y,1)=|x-y|^2\ge 0$. 
Since $r_2>(1+c_0)^{1/2}$, $|\mathcal Q(x,y,\cos s)|=|(r_1-\cos s)(r_2-\cos s)|\gtrsim (1-\cos s)$. Using \eqref{ph-d}, we see  
 \[ |\partial_s\cP(x,y,s)|\gtrsim 1.\]
 The same lower bound holds if $\inp xy<- (1+c_0)^{1/2}$. In fact,  $r_1, r_2\le -1$ because $\mathcal Q(x,y,-1)=|x+y|^2\ge 0$.  Therefore, van der Corput lemma gives $|\mathfrak P_\lambda[\psi_j](x,y)| \lesssim 2^{\frac{d}2j}\min(\lambda^{-1},\,2^{-j})\lesssim 2^{\frac{d-1}2j}\lambda^{-\frac12}.$ This completes the proof.
\end{proof}

\section{ Estimate  away from $\sqrt{\lambda}\mathbb S^{d-1}$: Proof of Theorem \ref{thm-locest}}\label{sec:local}
In this section we prove Theorem \ref{thm-locest} and show  the failure of the estimate \eqref{est-loc0} for $(1/p,1/q)\in [\mathfrak C,\mathfrak  D]\cup[\mathfrak C', \mathfrak  D']$. This and the lower bounds on   $\|\chi_{\B_\lambda}\Pi_\lambda\chi_{\B_\lambda}\|_{p\to q}$ in Proposition \ref{lower-mu} below show that the bounds in Theorem \ref{thm-locest} are sharp.

\subsection{Proof of Theorem \ref{thm-locest}}
Making use of  Lemma \ref{sym} and \ref{imply1}, we see that  it is sufficient to consider   
\[ \textstyle \tilde \Pi_\lambda:=  \sum_{0\le j\le j_\circ} \Pi_\lambda[\psi_j] \]
in place of $\Pi_\lambda$  where  $2^{j_\circ}\sim \lambda$. That is to say, 
the same estimates  hold for $\tilde \Pi_\lambda$ as  those for $\Pi_\lambda$ in Theorem \ref{thm-locest}.

Since 
$|\cD(x,y)|\ge 1/2$ for $x,y\in \B$, we have the estimate \eqref{che}. Applying Lemma \ref{tt-st},  we get 
\begin{equation} 
\label{basic} 
\begin{aligned} 
   \|   \chi_{\B_\lambda} \Pi_\lambda [\psi_j ] \chi_{\B_\lambda} \|_{p\to q}&\lesssim  \lambda^{-\frac12\dpq} 2^{(\frac{d+1}{2}\dpq-1)j},
\end{aligned} \end{equation}
provided that \ppqpair\,  is contained in the close quadrangle 
with vertices 
$(1/2,1/2)$,  $\mathfrak A$, $\mathfrak A'$ and $(1,0).$  
Thus, summation over $j$ gives
\Be
\label{eq:j-sum}
\textstyle \|    \chi_{\B_\lambda} \tilde  \Pi_\lambda  \chi_{\B_\lambda} \|_{p\to q }\lesssim  \lambda^{\beta(p,q)},
\Ee
for $\ppq\in \mathcal R_1\setminus[\mathfrak C, \mathfrak C']$ (see Figure \ref{local-type}).  It is convenient for our purpose to note  $\beta(1,2)=\frac {d-2}4$, $\beta(p,q)=\frac d2(\frac1p-\frac{d-1}{2d})- 1$ if $(1/p,1/q)\in(\mathfrak C,\mathfrak D]$, and $\beta(p,q)=-\frac{1}{d+1}$ if $(1/p,1/q)=\mathfrak C$. By  $(c)$ in Lemma \ref{s-trick} and \eqref{basic}   we have
\begin{equation}
\label{rest-weak}
   \textstyle  \|  \chi_{\B_\lambda} \tilde  \Pi_\lambda  \chi_{\B_\lambda} \|_{L^{p,1} \to L^{q,\infty} }\lesssim  \lambda^{-\frac12\dpq},      \end{equation} 
     for  $\ppq=\mathfrak C, \mathfrak C'$,  which satisfy $\dpq=2/(d+1)$. Interpolation yields \eqref{eq:j-sum} for $\ppq\in \mathcal R_1$. 
      Using the estimate \eqref{basic} and taking sum over $0\le j\le j_\circ$,  we  get  \eqref{eq:j-sum} for $(p,q)=(1,\infty).$

Now, in view of interpolation, to complete  the proof  we need only to show \eqref{eq:j-sum} with $(p,q)=(1,2)$ and 
the weak type  estimate 
\Be 
\label{weak-pq} \textstyle \|  \chi_{\B_\lambda} \tilde  \Pi_\lambda  \chi_{\B_\lambda} \|_{L^p\rightarrow L^{\frac{2d}{d-1},\infty}} \lesssim   \lambda^{ \frac d2(\frac1p-\frac{d-1}{2d})- 1}, 
\Ee
for $1\le p<{2d(d+1)}/(d^2+4d-1)$, which corresponds to the estimate $(i)$  in Theorem \ref{thm-locest} for $\ppq\in (\mathfrak C, \mathfrak D]$.
Duality and interpolation provide
 all the $L^p$--$L^q$ estimates  asserted  in Theorem \ref{thm-locest}  (Figure \ref{local-type}).

By the estimate \eqref{202} we have 
\begin{align}
   \| \Pi_\lambda [\psi_j] f\|_2 &\lesssim  2^{-\frac{j}{2}}\lambda^{\frac{d-2}{4}}\|f\|_1, \quad  2^j\lesssim \lambda.
   \label{12j}
\end{align}
Summation over $j$ clearly yields  \eqref{eq:j-sum} with $(p,q)=(1,2)$.  It now remains to show \eqref{weak-pq}. 
Interpolation of  the estimates \eqref{12j} and   \eqref{basic} with $\ppq=(1,0),$ $\mathfrak C$ gives 
\Be
\label{inter}
    \| \chi_{\B_\lambda} \Pi_\lambda [\psi_j ] \chi_{\B_\lambda}\|_{p\rightarrow q}\lesssim2^{jd(\frac{d-1}{2d}-\frac1q)}\lambda^{ \frac d2(\frac1p+\frac1q)- \frac{d+1}{2}}
\Ee 
for  $(1/p,1/q)$ contained  in the closed triangle  $\mathscr T$ with vertices $(1, 1/2), \mathfrak C,(1,0)$. 
Now, fixing $p\in [1, {2d(d+1)}/{(d^2+4d-1)})$ and choosing two $q_0, q_1$ such that $ q_0<2d/(d-1)<q_1$ and $(1/p, 1/q_0), (1/p, 1/q_1)\in \mathscr T$,  we have the estimates 
\eqref{inter}  with $p=p_0=p_1$ and $q=q_0, q_1$. Then, we  apply $(a)$ in Lemma \ref{s-trick} to  these two estimates to get the weak type estimate for 
$p,q$  such that  $\ppq\in (\mathfrak C,\mathfrak D]$. (Figure \ref{local-type} is  helpful here.) Hence, for $1\le p<{2d(d+1)}/(d^2+4d-1)$, we obtain   the estimate \eqref{weak-pq}.
This completes the proof of Theorem \ref{thm-locest}.

\begin{rem}
As can be easily seen from the proof,  the same bounds remain to hold on  $\che \Pi_\lambda \che$ $($in place of $\chi_{\B_{\lambda}}\Pi_\lambda\chi_{\B_{\lambda}}$$)$
as long as $E$ is a measurable set $\subset \R^d$ such that $E = -E$ and
    \[
    \mathcal D(x,y) \ge c_0, \quad  \forall\,   x,y\in E
    \]
for some constant $c_0>0$.  The same condition was used in  \cite{LR} to study $L^p$ boundedness of Bochner-Riesz means of the Hermite expansion. 
\end{rem}

\subsection{Boundedness of  the operator  $\wp_k$ and a transplantation result 
}\label{transplantation}
In this section we  consider the estimate  \eqref{est-loc0} 
for $\ppq\in\mathcal R_3$ and discuss how it is related to its counterpart to the Laplacian, that is to say, the estimate for $\proj$. 

To put our discussion in a proper context, recalling \eqref{def-pk}, we  consider the estimate 
\begin{align}\label{scaled-euclid} \big\|\proj   \big\|_{p\to q}\lesssim k^{-1+\frac{d}{2}\dpq} 
\end{align}
for $k\ge 1$.  The following lemma shows that  \eqref{scaled-euclid}  is equivalent to the estimate
\begin{equation}\label{est:euclid}
    \|( \widehat f\, |_{ \mathbb S^{d-1}})^\vee\|_q\lesssim \|f\|_p.
\end{equation}
\begin{lem} 
\label{equivalence}
The  estimate
\eqref{scaled-euclid} holds for all $k\ge 1$  if and only if the estimate  \eqref{est:euclid}  holds for all $f\in \mathcal S(\mathbb R^d)$.  
The equivalence remains valid with  $L^p$, $L^q$ spaces replaced, respectively,  by the Lorentz spaces  $L^{p,r}$, $L^{q,s}$ if $1<p,q\le \infty$ and $1\le r, s\le \infty$.  
\end{lem}

\begin{proof}
We consider
  \[ \wproj  f=\frac k{(2\pi)^d}\int_{\sqrt{1-1/k} \,\le |\xi| <1}e^{ix\cdot\xi}\,\widehat f(\xi) d\xi.\]  
By scaling we note that  $\| \proj\|_{p\to q} =k^{\frac d2 \dpq-1}\| \wproj \|_{p\to q}$. Thus,  the estimate  \eqref{scaled-euclid}  is  equivalent to 
 \begin{align}\label{est:saturated} 
 \big\| \wproj \big\|_{p\to q} \lesssim   1. 
 \end{align}
  Letting $k\to \infty$ gives  \eqref{est:euclid}. Conversely, making use of the spherical coordinates\footnote{ We write $\int_{\sqrt{1-1/k}\le |\xi|<1}\,e^{ix\cdot\xi}\,\widehat f(\xi) d\xi =C_d\int^{1}_{\sqrt{1-1/k}}\int e^{ix\cdot r\omega}\,\widehat f(r\omega) d\omega\, r^{d-1} dr$.} and Minkowski's inequality, one can easily  see that \eqref{est:euclid} 
 implies \eqref{est:saturated} and then \eqref{scaled-euclid} via scaling.  Extension to the Lorentz spaces is clear since $L^{p,r}$ is a Banach space if $1<p \le \infty$ and $1\le r\le \infty$. 
 We omit the detail. 
 \end{proof}

 The operator $f\to ( \widehat f\, |_{ \mathbb S^{d-1}})^\vee$ is imbedded in a family of operators which are called the Bochner-Riesz operators of negative order: 
\[ S^{\alpha}\!f(x)=(2\pi)^{-d}\int e^{ix\cdot\xi}\,\frac{(1-|\xi|^2)_+^\alpha}{\Gamma(\alpha+1)}\widehat f(\xi) d\xi, \] 
where $\Gamma$ is the gamma function. 
For $\alpha\le -1$, the operator is defined by analytic continuation of  the distribution ${(1-|\xi|^2)_+^\alpha}/{\Gamma(\alpha+1)}$. In fact, $S^{-1} f=( \widehat f\, |_{ \mathbb S^{d-1}})^\vee$. $L^p$--$L^q$ boundedness of $S^\alpha$ was studied by various authors \cite{B86, CS88, sogge-1986, bmo, bak, Gu99, CKLS}.  
The necessary  conditions  on $p,q$ for $L^p$--$L^q$ boundedness was shown by  B\"orjeson \cite{B86}.  In $\mathbb R^2$, the  problem  is settled by  Bak  \cite{bak} but  it remains open  for $d\ge 3$.  For the most recent development, see \cite{KL19}.  However,  we have a complete characterization of $L^p$--$L^q$ boundedness of  the operator $S^{-1}$.

\begin{thm}[\cite{B86, bmo, bak,  Gu99}]  \label{rr*} The operator  $f\to ( \widehat f\, |_{ \mathbb S^{d-1}})^\vee$  is bounded from $L^p$  to $L^q$ if and only if $\ppq\in \mathcal R_3$. 
Furthermore,  we have  $\| ( \widehat f\, |_{ \mathbb S^{d-1}})^\vee\|_{q,\infty}\lesssim \|f\|_p$ if $\ppq\in (\mathfrak C, \mathfrak D]$,\footnote{Duality gives $\| ( \widehat f\, |_{ \mathbb S^{d-1}})^\vee\|_{q}\lesssim \|f\|_{p,1}$ if $\ppq\in (\mathfrak C', \mathfrak D']$.} 
 and 
 $\| ( \widehat f\, |_{ \mathbb S^{d-1}})^\vee\|_{q,\infty}\lesssim \|f\|_{p,1}$  if $\ppq\in \mathfrak C, \mathfrak C'$. 
\end{thm}

 Using  Theorem \ref{rr*},  Lemma \ref{equivalence}, and the Stein-Tomas theorem, we  can obtain the sharp  $L^p$--$L^q$ estimate for $\proj$.

 \begin{cor}  
\label{laplacian} 
Let $d\ge 2$ and $(1/p,1/q)\in\sq$. Then, we have 
\begin{align}
\label{st} &\|\proj \|_{p\to q}\sim k^{\beta(p,q)},  && (1/p,1/q)\not\in [\mathfrak C,\mathfrak  D]\cup[\mathfrak C', \mathfrak  D'],
 \\
 \label{st0} &\|\proj\|_{L^p\to L^{q,\infty}}\le C k^{\beta(p,q)},  && (1/p,1/q)\in (\mathfrak C, \mathfrak D], 
 \\
\label{st1} &\|\proj\|_{ L^{p,1}\to L^{q,\infty}} \le C  k^{\beta(p,q)}, && (1/p,1/q)=\mathfrak C, \mathfrak C'. 
\end{align}
   \end{cor}
 
\begin{proof}
 Combining Theorem \ref{rr*}  and Lemma \ref{equivalence}, we get 
 the estimate \eqref{scaled-euclid} for  $\ppq\in \mathcal R_3$, including 
 the weak type estimate \eqref{st0} and the restricted weak type  estimate  \eqref{st1} 
 for $\ppq\in (\mathfrak C, \mathfrak D]$ and $\ppq=\mathfrak C', \mathfrak C$, respectively.  
So, we need only to show \eqref{st}. 
 
 We have $\|\proj\|_{2\to q}\lesssim   k^{\beta(2,q)}$  for $q={2(d+1)}/(d-1)$, which can  be shown  similarly as before,  using  the spherical coordinates, the Stein-Tomas theorem, and Plancherel's theorem.  The estimate 
 $\|\proj\|_{2\to \infty}\lesssim   k^{\beta(2,\infty)}$ follows from the Cauchy-Schwarz  inequality and Plancherel's theorem. These two estimates respectively correspond to the points   
$\mathfrak A'$ and $(1/2,0)$ in  Figure \ref{local-type} and then  duality gives the estimates for $\ppq=\mathfrak A$, and $(1,1/2)$. Since we have \eqref{st} for  $\ppq\in \mathcal R_3$, and \eqref{st0} and \eqref{st1}  together with $\|\proj  \|_{2\to 2}\lesssim  1$, interpolation and duality give 
\[ \|\proj\|_{p\to q} \lesssim  k^{\beta(p,q)}, \quad \ppq\not\in [\mathfrak C,\mathfrak  D]\cup[\mathfrak C', \mathfrak  D'].\]

 The opposite inequality can be easily shown. Since $\| \proj\|_{p\to q} =k^{\frac d2 \dpq-1}\| \wproj \|_{p\to q}$, by duality  we need only to show 
  \[\|\wproj\|_{p\to q} \gtrsim 
  \max( k^{1-\frac{d+1}2 \dpq}, 1, k^{  \frac{d}{q} -\frac{d-1}{2}}).\] 
  The second lower bound is trivial. Since the multiplier of the operator $\wproj$ is radial and supported in $O(k^{-1})$-neighborhood of the sphere $\mathbb S^{d-1}$,  the first and the third lower bounds  can be shown  by using, respectively,  a Knapp type example and the asymptotic expansion of the Bessel function (for example, see \cite{B86}).  This completes the proof of \eqref{st}. 
\end{proof}

 The following shows that the estimate \eqref{est-loc0} implies  \eqref{est:euclid} when $(1/p, 1/q)\in \mathcal R_3$.

\begin{lem} \label{implication} Let $B$ be a ball of small radius centered at the origin.  
 Suppose 
\Be\label{lplq-weaker}
\|\chi_{B}\Pi_{\lambda}\chi_{B}\|_{p\to q}\lesssim\lambda^{\frac d2\delta(p,q)-1}
\Ee 
holds. Then we have the estimate \eqref{est:euclid}.\end{lem}

By this and Theorem \ref{rr*} it follows that  \eqref{est-loc0} holds if and only if $(1/p,1/q)\in \mathcal R_3$.  

Lemma \ref{implication} may be compared with  the known  fact   \cite{Th87, KST} (\cite{LR})  that a local  $L^p$  bound on the Hermite Bochner-Riesz means implies an $L^p$  bound on the classical  Bochner-Riesz means.  Our proof  below is similar to that  in \cite{KST}, where transplantation of $L^p$ bounds for differential operators was proved. However, unlike $L^p$ bound,  $L^p$--$L^q$ estimate ($p\neq q$) is not scaling invariant. The particular form of the bound \eqref{lplq-weaker} plays a crucial role.   Our argument also extends to general second order elliptic operators without difficulty as long as the associated spectral projection operator satisfies the same form of  bound. 

In order to prove Lemma \ref{implication}, we recall  a special case of 
H\"ormander's result  \cite[Theorem 5.1]{H68}. 

\begin{thm}\label{thm:hor}Let $P$ be a self-adjoint elliptic differential operator of order 2 with $C^{\infty}$-coefficients
on $\mathbb R^d$  and $p$ be its principal part. Then, for $x,y$ in a compact subset and sufficiently close to each other, we have \begin{align*}
 \Big| e(x,y,\lambda)-(2\pi)^{-d}\int_{p(y,\xi)<\lambda}e^{i\psi(x,y,\xi)}d\xi\Big|\le C(1+|\lambda|)^{\frac{d-1}{2}},
 \end{align*}
 with $C$ independent of $\lambda$ 
 where $e(x,y,\lambda)$ is the spectral function of $P$, i.e., the kernel of the spectral projection operator $\Pi_{[0,\lambda]}$,\footnote{Here, the operator $\Pi_{[0,\lambda]}$ is defined by  the typical spectral resolution.}   and $\psi$ is a function  homogeneous in $\xi$ of degree 1 which satisfies $ p(x,\nabla_x\psi)=p(y,\xi)$ and 
 \begin{align}
 \label{spectral-phase} \psi(x,y,\xi)=\left\langle x-y,\xi\right\rangle+O(|x-y|^2|\xi|).
 \end{align}
 \end{thm}

\begin{proof}[Proof of Lemma \ref{implication}]
 Let $k, \nu$ be large positive integers. Consider an auxiliary  projection operator  
\[ \textstyle \widetilde \Pi=\sum_{k\nu< \lambda \le(k+1)\nu}\Pi_\lambda .\] 
 By the triangle inequality and the assumption \eqref{lplq-weaker} we have 
\[
\textstyle
\| \chi_{ B} \widetilde \Pi \chi_{B}\|_{p\rightarrow q}\le\sum_{k\nu<\lambda\le(k+1)\nu}\|\chi_{ B}\Pi_\lambda\chi_{B}\|_{p\to q}\lesssim k^{-1}(k\nu)^{\frac{d}{2}(\frac1p-\frac1q)}.
\] 
Let $f,g$ be nontrivial functions in $C^{\infty}_c(\mathbb{R}^d)$ such that $\supp f$, $\supp g\subset B$.
Since $\inp{\widetilde \Pi f}{g}=\inp{\chi_{B} \widetilde \Pi\chi_{ B}f}{g}$, we have
\begin{align*}
 \Big|\iint \widetilde \Pi(x,y)f(x)g(y)dxdy\Big|\lesssim k^{-1}(k\nu)^{\frac{d}{2}(\frac1p-\frac1q)}\|f\|_{p}\|g\|_{q'}
 \end{align*}
  Rescaling $(x,y)\to (\nu^{-1/2}x,\nu^{-1/2}y)$ gives the equivalent estimate 
 \begin{align}\label{est:trans} 
 \Big|\nu^{-\frac{d}{2}}\iint \widetilde \Pi(\nu^{-\frac12}x,\nu^{-\frac12}y)F(x) G(y) dxdy\Big|\lesssim k^{-1+\frac{d}{2}(\frac1p-\frac1q)} \|F\|_p \|G\|_{q'}
 \end{align}
 provided that $F$ and $G$ are supported in $\sqrt\nu B$.  Taking the radius of $B$ small enough, we may apply Theorem \ref{thm:hor}.
 Since $\widetilde \Pi(\nu^{-1/2}x,\nu^{-1/2}y) 
       = e(\nu^{-1/2}x,\nu^{-1/2}y,(k+1)\nu)-e(\nu^{-1/2}x,\nu^{-1/2}y, k\nu)$, 
 by Theorem \ref{thm:hor} we have
 \begin{align*} 
 \widetilde \Pi(\nu^{-\frac12}x,\nu^{-\frac12}y) 
       &=(2\pi)^{-d}\int_{k\nu\le |\xi|^2<(k+1)\nu}e^{i\psi(\nu^{-\frac12}x,\nu^{-\frac12}y,\xi)}d\xi+R_{\nu,k}(x,y),
 \end{align*}
 where $R_{\nu,k}(x,y)=O(|k\nu|^{\frac{d-1}{2}})$.
Changing variables $\xi\to \nu^{1/2}\xi$ gives 
\begin{align*} 
 \widetilde \Pi(\nu^{-\frac12}x,\nu^{-\frac12}y) 
      =(2\pi)^{-d}\nu^\frac d2\int_{k\le  |\xi|^2 <(k+1)}e^{i\psi(\nu^{-\frac12}(x,y),\nu^\frac12 \xi)}d\xi+R_{\nu,k}(x,y).
 \end{align*}
  Combining this and \eqref{est:trans}  yields 
 \begin{align*} 
 \Big|\iint\Big(\int_{k\le  |\xi|^2<k+1}\hspace{-25pt} e^{i\psi(\nu^{-\frac12}(x,y),\nu^{\frac12}\xi)}d\xi+\widetilde R_{\nu,k}(x,y)\Big)f(x)g(y)dxdy\Big|\lesssim k^{-1+\frac{d}{2}(\frac1p-\frac1q)}
 \end{align*}
 whenever  $f$ and $g$ are supported in $\sqrt\nu B$ and $\|f\|_p=\|g\|_{q'}=1$. 
 Here $\widetilde R_{\nu,k}=O( \nu^{-\frac{d}{2}}(k\nu)^{\frac{d-1}{2}})$. From \eqref{spectral-phase}, note that the phase function $\psi(\nu^{-1/2}(x,y),\nu^{1/2}\xi) \to  \langle x-y,\xi\rangle$ as $\nu\to \infty$. Thus, taking $\nu\rightarrow\infty$, we obtain
 \begin{align*} 
 \Big|\iint\Big(\int_{k\le  |\xi|^2<k+1}e^{i\left\langle x-y,\xi\right\rangle}d\xi\Big)f(x)g(y)dxdy\Big|\lesssim k^{-1+\frac{d}{2}(\frac1p-\frac1q)}
 \end{align*}
 if $\|f\|_p=\|g\|_{q'}=1$.   This gives  \eqref{scaled-euclid},
 which is equivalent to \eqref{est:euclid} as seen above. 
\end{proof}

 \section{ Estimate  near $\sqrt{\lambda}\mathbb S^{d-1}$: Proof of Theorem \ref{thm-annest}}\label{sec:amu}
In this section we prove Theorem \ref{thm-annest}. To this end, it is more convenient to consider the rescaled operator $\fP_\lambda$ instead of $\Pi_\lambda$. 
We note that 
\begin{equation}
\label{eq:norm-scaled}
  \|  \chi_{E} \mathfrak \fP_\lambda [\eta]  \chi_{E} \|_{L^{p,r}\to L^{q,s}} =  \lambda^{\frac d2(\frac1p-\frac1q-1)}\|  \chi_{\sqrt\lambda E} \Pi_\lambda [\eta] \chi_{\sqrt\lambda E}\|_{L^{p,r}\to L^{q,s}}, 
   \end{equation}
 for any measurable set $E \subset \R^d$ where  $\sqrt\lambda E:=\{x : \lambda^{-1/2}x\in E\}.$
The key part of  the proof is to show the following. 

\begin{prop}
\label{hoh} Let $\lambda^{-2/3}\le\mu\le 1/4$ and $2^j<\lambda\mu$. Then, we have 
\begin{align}
    \label{weak2}
   & \textstyle \|\sum_{2^j<\lambda \mu}  \chi_\mu\fP_\lambda[\psi_j]\chi_\mu\|_{L^{\frac{2d}{d+1},1}\to L^\infty}\lesssim \lambda^{-\frac12}\mu^{\frac{d-3}{4}}, \\
    &\label{scaled5/6}
    \|\chi_{\mu}\fP_\lambda[\psi_j]\chi_{\mu}\|_{L^{6/5,1}\to L^{6,\infty}}\lesssim \lambda^{-\frac{d+2}{6}} \mu^{-\frac13} 2^{\frac{d-2}{3}j}, \quad j\ge 0.
\end{align}
\end{prop} 

Once we have the above estimates,  the assertions in Theorem \ref{thm-annest} can easily be verified.

\begin{proof}[Proof of Theorem \ref{thm-annest}] To prove Theorem \ref{thm-annest}, it suffices to show \eqref{rweak} and  \eqref{weakann}. 
Indeed, note that $\|\chi_{\lambda,\mu} \Pi_\lambda\chi_{\lambda,\mu}\|_{2\to \infty}\le  C\lambda^{(d-2)/4}\mu^{(d-1)/4}$. This follows  
by \eqref{kt-nu}  since 
$\|\chi_{\lambda,\mu} \Pi_\lambda\chi_{\lambda,\mu}\|_{2\to \infty}\le  \|\chi_{\lambda,\mu}  \Pi_\lambda\|_{2\to \infty} \|\chi_{\lambda,\mu} \Pi_\lambda\|_{2\to 2}.$ 
The desired estimate \eqref{est-ann} for $(1/p,1/q)\in(\mathfrak L_1\cup\mathfrak L_2\cup\mathfrak L_2'\cup\mathfrak L_3)$ follows from those estimates and the previously known estimate \eqref{kt-nu} (equivalently, \eqref{ktdia}) via interpolation and duality (see Figure \ref{xyx} and \ref{xyxz}).

Thanks to \eqref{eq:norm-scaled},    \eqref{rweak} follows from  \eqref{weak2}  and Lemma \ref{imply1}.  Similarly,  for \eqref{weakann}  it is enough to show 
\Be\label{weakann11}
   \textstyle  \| \sum_{2^j<\lambda \mu} \chi_{\lambda,\mu} \Pi_\lambda[\psi_j] \chi_{\lambda,\mu} \|_{L^{p,1}\to L^{q,\infty}} \lesssim (\lambda \mu)^{-\frac{1}{d+1}}, \quad ({1}/{p},{1}/{q})=\mathfrak G, \mathfrak G'. 
\Ee 
When $d=2$,  there is nothing to prove since $L^p$--$L^q$ estimate holds by \eqref{ktdia}.  Thus,  we may assume $d\ge 3$.  To do this, we use Lemma \ref{tt-st} and  Lemma \ref{s-trick}.  By scaling, i.e., \eqref{eq:norm-scaled}, the estimate \eqref{scaled5/6} is equivalent to 
\[ 
    \|\chi_{\lambda,\mu}\Pi_\lambda[\psi_j]\chi_{\lambda,\mu}\|_{L^{6/5,1}\to L^{6,\infty}}\lesssim (\lambda\mu)^{-\frac13} 2^{\frac{d-2}{3}j}, \quad j\ge 0.
\]
Since $d\ge 3$, the exponent of  $2^j$ is positive. So, we can apply Lemma \ref{tt-st} with $r=6/5$, $\beta = (\lambda\mu)^{-1/2}$, and $b = {(d-1)/}2$ to get
\[\|\chi_{\lambda,\mu}\Pi_\lambda[\psi_j]\chi_{\lambda,\mu}\|_{p\to q}\lesssim (\lambda\mu)^{-\frac32\delta(p,q)} 2^{j(-1+\frac{d+1}{2}\delta(p,q))}  \] 
for $(1/p,1/q)\in \mathfrak Q(\frac{d-1}2,\frac65)$. Applying $(c)$ in Lemma \ref{s-trick} yields \eqref{weakann11} for $(1/p,1/q)=\mathfrak G, \mathfrak G'$.
\end{proof}

\subsection{Reduction via sectorial decomposition} 
We prove the estimates  \eqref{weak2} and \eqref{scaled5/6} while assuming  under the assumption that 
\[ \mu\ll 1.\]
The case $\mu\sim 1$ can be handled in a similar way but much easier  (see Remark \ref{frmk}). 

 We make use of a decomposition of $A_\mu\times A_{\mu}$, which was used in  \cite{JLR_endpoint}.  
Note that 
\begin{align}
 \label{angle} 
\mathcal D(x,y)=-|x|^2|y|^2 \sin^2\theta(x,y) +(1-|x|^2)(1-|y|^2), 
 \end{align}
 where   $\theta(x,y)$ denotes  the angle  between $x$ and $y$. 
Since $|(1-|x|^2)(1-|y|^2)|\sim \mu^2$ for $(x,y)\in A_\mu\times A_{\mu}$, relative size of  $\theta(x,y)$  against $\mu$ is efficient to control  $\cD$.  This can be exploited by  a Whitney type decomposition of $\mathbb S^{d-1}\times \mathbb S^{d-1}$ away from its diagonal (see   \cite[Section 2.4]{JLR_endpoint}). 

Adopting the typical dyadic decomposition process, for each integer $\nu\ge 0$ we partition
$\mathbb S^{d-1}$ into spherical caps $\Theta_k^\nu$  such that  $\Theta_k^\nu \subset \Theta_{k'}^{\nu'}$ for some $k'$ whenever $\nu\ge \nu'$,  
and $c_d2^{-\nu}\le 
\diam (\Theta_k^\nu)  \le C_d2^{-\nu}$ for some constants $c_d$, $C_d>0$. Let  
$\nu_\circ:=\nu_\circ(\mu)$ denote the integer  $\nu_\circ$ such that 
\[  
\mu/2<   C2^{-\nu_\circ} \le \mu
\]
for a large positive constant $C$.
By a Whitney type decomposition of $\mathbb S^{d-1}\times \mathbb S^{d-1}$ away from its diagonal, we may write
\[\textstyle   \mathbb S^{d-1}\times \mathbb S^{d-1}=\bigcup_{ \nu\in \mathbb N_0:  2^{-\nu_\circ} \le    2^{-\nu} \lesssim 1 }\,\, \bigcup_{k\sim_\nu k'} \Theta_k^\nu\times \Theta_{k'}^{\nu},\] 
where  $k\sim_\nu k'$ implies $\dist (\Theta_k^{\nu}, \Theta_{k'}^{\nu})\sim 2^{-\nu}$ if  $\nu> \nu_\circ$ and  $\dist (\Theta_k^{\nu}, \Theta_{k'}^{\nu})\lesssim 2^{-\nu}$ if  $\nu= \nu_\circ $ (for example, see \cite[p.971]{TVV}). It should be noted that  the sets $\Theta_k^{\nu_\circ}$ and $\Theta_{k'}^{\nu_\circ}$  are not necessarily  separated at $\nu=\nu_\circ$.  For a fixed $\mu$  we define 
\[  A_{k}^{\nu}=\big\{ x\in A_{\mu}:  |x|^{-1}x\,\in    \Theta_{k}^\nu\big\}\] 
and set $\chi_{k}^{\nu}=\chi_{A_{k}^{\nu}}$. 
Thus we can write  
\Be
\label{s-decom}
\chi_\mu \fP_\lambda[\psi_j]\chi_\mu = \sum_{2^{-\nu_\circ}\le     2^{-\nu} \lesssim 1}\,\,\sum_{k\sim_\nu k'}\chi_{k}^{\nu}\fP_\lambda[\psi_j]\chi_{k'}^{\nu}.
\Ee

The following simple lemma basically reduces the estimate for $\sum_{k\sim_\nu k'}\chi_{k}^{\nu}\fP_\lambda[\psi_j]\chi_{k'}^{\nu}$ to a uniform estimate for $\chi_{k}^{\nu} \fP_\lambda [\psi_j] \chi_{k'}^{\nu} $ with $k\sim_\nu k'$.

\begin{lem}\label{kkpsum}  Let $1\le p\le q\le \infty$ and let $T$ be an operator from $L^{p,1}$ to $L^{q,\infty}$. 
 Suppose we have the estimate 
$\|\chi_{k}^{\nu}  \,  T  \chi_{k'}^{\nu} \|_{L^{p,1} \to L^{q,\infty}}\le B$ whenever $k\sim_\nu k'$. Then, with $C$ only depending on $d$,   we have 
$
\|\sum_{k\sim_\nu k'} \chi_{k}^{\nu} \, T  \chi_{k'}^{\nu} \|_{ L^{p,1} \to L^{q,\infty}}\le C B. 
$
The same also holds when $L^{p,1}$ and $L^{q,\infty}$ are replaced by $L^{p}$ and $L^{q}$. 
\end{lem}

\begin{proof} 
The last assertion is clear. We only provide the proof   for $L^{p,1}$ and $L^{q,\infty}$.  It is enough to show $\|\sum_{k\sim_\nu k'} \chi_{k}^{\nu} \, T  \chi_{k'}^{\nu}\chi_F \|_{ L^{q,\infty}}\le C B \|\chi_F\|_p$ for any measurable set $F$ (e.g., see Stein \cite[p.195]{St71}). 
Besides, note that $\|\sum_k f_k\|_{q,\infty}\le ( \sum_{k} \|f_k\|_{q,\infty}^q)^{1/q}$ if $\supp f_k$ are disjoint. By combining those factors, one can easily see 
the desired inequality since $\supp \chi_{k}^{\nu}$ are boundedly overlapping. 
\end{proof}

To obtain the desired estimates, we separately consider the cases $2^{-\nu}\gg\mu$ and  $2^{-\nu}\lesssim  \mu$ for which  we have $|\cD(x,y)|\gtrsim 2^{-2\nu}$ and $|\cD(x,y)|\lesssim \mu^2$, respectively.

\subsubsection*{When $2^{-\nu}\gg\mu$}  
Since $-\cD(x,y)\sim 2^{-2\nu}$ on $A_k^\nu\times A_{k'}^\nu$, substituting $N=1/2$ in $(a)$  (or the last assertion if $2^{-\nu}\sim 1$) of Lemma \ref{oscillatory}, we have
\Be\label{1inf_easy}
\|\chi_{k}^{\nu} \fP_\lambda [\psi_j] \chi_{k'}^{\nu}\|_{1\to \infty}\lesssim \lambda^{-\frac12}2^{\frac{d-1}{2}j}2^{\frac{\nu}{2}}.
\Ee
By  \eqref{201} and \eqref{eq:norm-scaled} it follows that $
\|\chi_k^\nu\fP_\lambda[\psi_j]\chi_{k'}^\nu\|_{2\to\infty}\lesssim 2^{-\frac j2}\lambda^{-\frac12}\mu^{\frac{d-2}{4}}.$ 
Combining this estimate and \eqref{1inf_easy} with  Lemma \ref{s-trick}, we obtain
\[
\tx \|\sum_{j\ge 0}\chi_k^\nu\fP_\lambda[\psi_j]\chi_{k'}^\nu\|_{L^{\frac{2d}{d+1},1}\to L^\infty}\lesssim \lambda^{-\frac12}\mu^{\frac{d^2-3d+2}{4d}}2^{\frac{\nu}{2d}}.
\]
Now,  Lemma \ref{kkpsum} and summation over $\nu : \mu\ll 2^{-\nu}\lesssim 1$ give 
\Be\label{rest_easy}
\tx
\sum_{\mu\ll 2^{-\nu}\lesssim 1} \|\sum_{j\ge 0}\sum_{k\sim_\nu k'}\chi_k^\nu\fP_\lambda[\psi_j]\chi_{k'}^\nu\|_{L^{\frac{2d}{d+1},1}\to L^\infty}\lesssim\lambda^{-\frac12}\mu^{\frac{d-3}{4}}.
\Ee

By \eqref{easy-l2} and \eqref{eq:norm-scaled}, we have $\|\chi_{k}^{\nu}\fP_\lambda[\psi_j]\chi_{k'}^{\nu}\|_{2\to 2}\lesssim \lambda^{-\frac d2}2^{-j}$. 
Interpolation with \eqref{1inf_easy} gives the estimate $\|\chi_{k}^{\nu}\fP_\lambda[\psi_j]\chi_{k'}^{\nu}\|_{6\to 6/5}\lesssim \lambda^{-\frac{d+2}{6}}2^{\frac{d-2}{3}j} 2^{\frac\nu 3}$. 
Using Lemma \ref{kkpsum} and taking sum over $\nu : \mu\ll 2^{-\nu}\lesssim 1$, we obtain 
\Be\label{rweak_easy}
\textstyle \sum_{\mu\ll2^{-\nu}\lesssim 1} \big\|\sum_{k\sim_\nu k'} \chi_k^\nu \fP_\lambda[\psi_j]\chi_{k'}^\nu\big\|_{L^{6/5}\to L^{6}}\lesssim \lambda^{-\frac{d+2}{6}}2^{\frac{d-2}{3}j} \mu^{-\frac13}.
\Ee

\subsubsection*{When $2^{-\nu}\lesssim \mu$} 
We now consider the case $2^{-\nu_0}\le 2^{-\nu}\lesssim \mu$. From \eqref{rweak_easy} and \eqref{rest_easy} we note  that the contributions in this case  $2^{-\nu}\gg\mu$ are acceptable to  
the estimates \eqref{weak2} and \eqref{scaled5/6}.  
Since there are only $O(1)$ $\nu$, to show \eqref{weak2} and \eqref{scaled5/6} it is sufficient to consider a single $\nu$ such that  $2^{-\nu}\lesssim \mu$. By Lemma \ref{kkpsum} we need only have to  show  the estimates
\begin{align*}
&\|\chi_k^\nu \fP_\lambda[\psi_j]\chi_{k'}^\nu\|_{L^{6/5,1}\to L^{6,\infty}}\lesssim \lambda^{-\frac{d+2}{6}}\mu^{-\frac13}2^{\frac{d-2}{3}j}, 
\\
&\tx\|\sum_{j\ge 0} \chi_k^\nu \fP_\lambda[\psi_j]\chi_{k'}^\nu\|_{L^{\frac{2d}{d+1},1}\to L^\infty}\lesssim \lambda^{-\frac12}\mu^{\frac{d-3}{4}}
\end{align*}
for $k\sim_\nu k'$.  Note that $A_{k}^{\nu}$ and $A_{k'}^{\nu}$ are contained in a set of diameter $\sim \mu$ if   $k\sim_\nu k'$.  
Thus, for  $(x,y)\in  A_{k}^{\nu}\times A_{k'}^{\nu}$,   $k\sim_\nu k'$, we have  
\Be 
\label{inp} 1-\inp xy\sim \mu. 
\Ee

Let $\vepc, c\ll 1$ be positive constants which are to be specified later.  For further reduction we cover $A_{k}^{\nu}$ and $A_{k'}^{\nu}$ by collections of  essentially disjoint cubes 
$\{Q\}$ and $\{Q'\}$ of side length  $c\vepc\mu$, respectively, so that 
\[ A_{k}^{\nu}\subset \bigcup Q, \quad  A_{k'}^{\nu}\subset \bigcup Q' .\] 
Note that $\partial_x\cD(x,y) = 2(\inp xy-1) y + 2 (y-x)$ and $\partial_y\cD(x,y) = 2(\inp xy-1) x + 2 (x-y).$ 
Since $1-\inp xy\sim \mu$ and $|x-y|\lesssim \mu$ for $(x,y)\in A_{k}^{\nu} \times A_{k'}^{\nu}$,  we have $\partial_x\cD$ and  $\partial_y\cD$ are $O(\mu)$. 
Thus, taking $c$ small enough,  we  have one of the following hold for each  $Q\times Q'$:
\begin{align}
\label{large-d}
  &|\mathcal D(x,y)|\gtrsim \vepc \mu^2, \  \  \  \forall (x,y)\in   \tilde Q\times \tilde Q',  
   \\ 
   \label{small-det}
     & |\mathcal D(x,y)|\ll \vepc  \mu^2, \  \  \  \forall (x,y)\in   \tilde Q\times \tilde Q',    
  \end{align} 
  where $ \tilde Q$ and $ \tilde Q'$ denote $c^2\vepc \mu$-neighorhoods of $Q$ and $Q'$, respectively. 

Since there are at most $O((c\vepc)^{-d})$ many $Q$ and $Q'$, the matter is reduced to  showing the following estimates for each $Q\times Q'$: 
\begin{align}
\label{goal1}
&\|\chi_Q \fP_\lambda[\psi_j]\chi_{Q'}\|_{L^{6/5,1}\to L^{6,\infty}}\lesssim \lambda^{-\frac{d+2}{6}}\mu^{-\frac13}2^{\frac{d-2}{3}j}, 
\\
\label{goal11}
&\tx\|\sum_{j\ge 0} \chi_Q \fP_\lambda[\psi_j]\chi_{Q'}\|_{L^{\frac{2d}{d+1},1}\to L^\infty}\lesssim \lambda^{-\frac12}\mu^{\frac{d-3}{4}}.
\end{align}
If  \eqref{large-d} holds, one can easily obtain the desired estimates  \eqref{goal1} and \eqref{goal11} by the same  argument as  above. Indeed, $(c)$ with $N=1/2$ in Lemma \ref{oscillatory}, Lemma \ref{PI-est}, and \eqref{easy-l2},   respectively, give 
\begin{align}
\label{10}
\|\chi_Q\fP_\lambda[\psi_j]\chi_{Q'}\|_{1\to\infty}&\lesssim 2^{\frac{d-1}{2}j}\lambda^{-\frac12}\mu^{-\frac12},
\\
\label{20}
 \|\chi_Q\fP_\lambda[\psi_j]\chi_{Q'}\|_{2\to\infty}&\lesssim 2^{-\frac j2}\lambda^{-\frac12}\mu^{\frac{d-2}{4}}, \quad 2^j\le \lambda\mu, 
 \\
 \label{22} 
 \|\chi_Q\fP_\lambda[\psi_j]\chi_{Q'}\|_{2\to 2}&\lesssim \lambda^{-\frac d2}2^{-j}.
\end{align}
Applying Lemma \ref{s-trick} to \eqref{10} and \eqref{20}, we obtain \eqref{goal11}. In the same manner, the estimate \eqref{goal1} follows by \eqref{10} and  \eqref{22}.

We now consider the case \eqref{small-det}. In this case, the estimates \eqref{20} and \eqref{22} remains valid. However, 
\eqref{10} holds only for $j$ such that $2^{-j}\ll\mu^{1/2}$ or $2^{-j}\gg\mu^{1/2}$
as can be seen by taking $N=1/2$ in $(b)$ of Lemma \ref{oscillatory}.  
Thus, repeating the same argument above, we obtain \eqref{goal1} for $2^{-j}\not\sim \mu^{\frac12}$ and 
\begin{align*}
\tx\|\sum_{2^{-j}\not\sim\mu^{1/2}} \chi_Q \fP_\lambda[\psi_j]\chi_{Q'}\|_{L^{\frac{2d}{d+1},1}\to L^\infty}\lesssim \lambda^{-\frac12}\mu^{\frac{d-3}{4}}. 
\end{align*}

Therefore, the proof of \eqref{goal1} and \eqref{goal11} is reduced to showing   
\begin{align}    
\label{goal21}
    &\|\chi_Q \fP_\lambda[\psi_j]\chi_{Q'}\|_{L^{6/5,1}\to L^{6,\infty}}\lesssim \lambda^{-\frac{d+2}{6}}\mu^{-\frac{d}{6}}, 
    \\
\label{goal22}    
    &\|\chi_Q \fP_\lambda[\psi_j]\chi_{Q'}\|_{L^{\frac{2d}{d+1},1}\to L^\infty}\lesssim \lambda^{-\frac12}\mu^{\frac{d-3}{4}}
    \end{align}
for $j$ satisfying $2^{-j}\sim \mu^{1/2}$ while assuming \eqref{small-det}.

\newcommand{\psic}{{\psi_{\mathbf c}}}

Let $\mathbf c_Q$ and $\mathbf c_{Q'}$ denote the centers of the cubes $Q$ and $Q'$, respectively. 
By $s_\mathbf c \in (0, \pi/2)$ we  denote the number such that $\cos s_\mathbf c = \inp{\mathbf c_Q}{\mathbf c_{Q'}}$, and set 
\[ \psi_{\mathbf c}(s)=\eta_\ast \Big( \frac{s-s_\mathbf c}{\sqrt{\vepc \mu}\,} \,\Big)\] 
where $\eta_\ast\in C_c^\infty((-2,2))$ such that $\eta_\ast=1$ on $[-1,1]$. Then, we decompose 
 \[\chi_Q\fP_\lambda[\psi_j]\chi_{Q'} = \chi_Q\fP_\lambda[\psi_{\mathbf c}\psi_j]\chi_{Q'} + \chi_Q\fP_\lambda[(1-\psi_{\mathbf c})\psi_j]\chi_{Q'}.\]
 
It is easy to show that $\chi_Q\fP_\lambda[(1-\psi_{\mathbf c})\psi_j]\chi_{Q'}$ has  acceptable bounds. To this end, we recall \eqref{inp}  and  note  that $\inp xy=\cos s_\mathbf c + O(c\vepc\mu)$ for $(x,y)\in Q\times Q'$. Since  $|\cD(x,y)|\ll\vepc\mu^2$ and $c, \mu\ll 1$, using \eqref{q-def}, we see that 
\[ \mathcal Q(x,y,\cos s)\gtrsim \vepc \mu^2, \quad \forall (x,y)\in Q\times Q' \] 
if $s\in \supp ((1-\psi_{\mathbf c})\psi_j)$.   Via \eqref{ph-d} this lower bound gives $|\partial_s\phi_1(x,y,s)|\gtrsim \mu$ for $(x,y,s)\in Q\times Q'\times \supp ((1-\psi_{\mathbf c})\psi_j)$.  Thus, recalling \eqref{eq:scaled-op} and applying 
 van der Corput's lemma, we get 
 $|\fP_\lambda[(1-\psi_{\mathbf c})\psi_j](x,y)|\le (\lambda\mu)^{-1}\mu^{-\frac d4}$ for  $(x,y)\in Q\times Q'$.  
 Here we also use $|\partial_s^n (\mathfrak a (1-\psi_{\mathbf c})\psi_j)| \lesssim \mu^{-(n+d)/2}$. 
 Since $\mu\ge\lambda^{-\frac23}$, by the above bound we obtain  
  \[
\|\chi_Q\fP_\lambda[(1-\psi_{\mathbf c})\psi_j]\chi_{Q'}\|_{1\to\infty}\lesssim 2^{\frac{d-1}{2}j}(\lambda\mu)^{-\frac12},  \quad 2^{-j}\sim \sqrt \mu. 
\]
Meanwhile, by  \eqref{201} and \eqref{easy-l2}   we have the estimates  \eqref{22} and \eqref{20} with 
$\psi_j$ replaced by $(1-\psi_{\mathbf c})\psi_j$ when $2^{-j}\sim \sqrt \mu$. Interpolation shows that $\chi_Q\fP_\lambda[(1-\psi_{\mathbf c})\psi_j]\chi_{Q'}$ has the acceptable bounds.

Now, the proof of Proposition \ref{hoh} reduces to proving \eqref{goal21} and \eqref{goal22} with $\psi_j$ replaced by   $\psi_{\mathbf c}\psi_j$ assuming \eqref{small-det}. 
Before proceeding further, we replace $\chi_Q$ and $\chi_{Q'}$ with smooth functions  $\tilde \chi_Q$ and $\tilde \chi_{Q'}$, respectively, which are adapted to $Q$ and $Q'$. 
More precisely, $\tilde \chi_Q$ and $\tilde \chi_{Q'}$ satisfy that $\tilde \chi_Q\chi_Q=\chi_Q$, $\tilde \chi_{Q'} \chi_{Q'}=\chi_{Q'}$, $\partial^\alpha  \tilde \chi_Q$,  $\partial^\alpha \tilde \chi_{Q'}=O(\mu^{-|\alpha|})$, and  $\tilde \chi_Q$ and  $\tilde \chi_{Q'}$ are supported in  $\tilde Q$ and $\tilde Q'$ ($c^2\vepc\mu$-neighborhoods of $Q$ and $Q'$), respectively.   Now, the desired estimates follow from the next 
proposition.

\begin{prop}\label{prop-goal2} Let $j$ satisfy $2^{-j}\sim \mu^{1/2}$ and let $Q$ and $Q'$ be the cubes of side length $c\vepc \mu$  given as above. 
Suppose \eqref{small-det} holds. Then, we have 
    \begin{align}
    &\|\tilde \chi_Q \fP_\lambda[\psi_{\mathbf c}\psi_j]\tilde \chi_{Q'}\|_{L^{6/5,1}\to L^{6,\infty}}\lesssim \lambda^{-\frac{d+2}{6}}\mu^{-\frac{d}{6}},
    \label{go1} \\
    &\|\tilde \chi_Q \fP_\lambda[\psi_{\mathbf c}\psi_j]\tilde \chi_{Q'}\|_{L^{\frac{2d}{d+1},1}\to L^\infty}\lesssim \lambda^{-\frac12}\mu^{\frac{d-3}{4}}
    \label{go2}. 
    \end{align}
\end{prop}

We make some observations, which are to be useful in what follows: 
\begin{align}
\label{angle2} |\sin \theta (x,y)| \sim \mu,  \quad \forall (x,y)\in \tilde Q\times \tilde Q', 
\\
\label{x-y}
|x-y|\sim \mu, 
\quad \forall (x,y)\in \tilde Q\times \tilde Q'.
\end{align}
The first \eqref{angle2} follows by \eqref{angle} and  \eqref{small-det}  since $(1-|x|^2)(1-|y|^2)\sim \mu^2$ if $x,y\in A_\mu$.  To see the second \eqref{x-y}, 
note that $
|x-y|^2 = (1-\inp xy)^2 - \cD(x,y)$. Thus,  by \eqref{small-det} and \eqref{inp} we have  \eqref{x-y}. 

\subsection{2nd-order derivative of $\phi_1$}
To prove Proposition \ref{prop-goal2}, we can no longer rely only on  the first order derivative of $\phi_1$.
When the discriminant $\cD(x,y)$ vanishes, the equation $\mathcal Q(x,y,\cdot) = 0$ has a zero of order 2. 
Furthermore, the stationary point of $\phi_1$ and the zero of $\partial_s^2\phi_1$ converge to each orther  as $\cD$ approaches to zero.
Thus, van der corput's lemma gives a decay estimate of  $O(\lambda^{-1/3})$ when $\cD(x,y) = 0$. 
However, such a bound is not sufficient for us to obtain the sharp bound since  we need $L^1$--$L^\infty$ bound of $O(\lambda^{-1/2})$  to make our argument work.
To overcome this problem, we break the integral dyadically away from the zero of $\partial_s^2\phi_1$. Before doing so, we need to take a close look at $\partial_s^2\phi_1$. 

A computation shows 
\begin{equation}
\label{ph-dd}
 \partial_s^2 \phi_1(x,y,s) = - \frac{\mathcal R(x,y,\cos s)}{ \sin^3 s},
 \end{equation}
where 
\[ \mathcal R(x,y,\tau) =  \inp xy \tau^2-(|x|^2+|y|^2)\tau+ \inp xy.\]
From \eqref{angle2} we note that  $x\neq y$ and $x\neq -y$. Thus, $\mathcal R(x,y, \cdot) $ has two distinct roots 
\Be
\label{roots}  \tau^\pm(x,y) =\frac{|x|^2+|y|^2\pm |x+y||x-y|}{2\inp xy}.
\Ee
It is easy to see  $\tau^+(x,y)> 1>\tau^-(x,y)$, and hence the role of $\tau^-$ is more important.  $\partial_s^2\phi_1(x,y,\cdot)$ has a unique zero on $(0,\pi/2)$ for $(x,y)\in \tilde Q\times \tilde Q'$, which we denote by $S_c(x,y)$. 
That is to say,  
\[\cos S_c(x,y) = \tau^-(x,y).\]
As clear from \eqref{inp} and \eqref{roots}, $S_c$ is smooth on $\tilde Q\times \tilde Q'$.
Using $|x+y|^2-|x-y|^2=4\inp xy$, we also have
\begin{equation}
\label{sc}
1-\cos S_c(x,y)
= {2|x- y|}/{(|x+y|+|x-y|)}.
\end{equation} 
Since $|x+y|\sim 1$ for $(x,y)\in \tilde Q\times \tilde Q'$,  $1-\cos S_c(x,y) \sim |x-y|\sim \mu$ and $S_c(x,y)\sim \mu^{1/2}$.

Let $s_*(x,y)\in (0, \pi/2)$ denote the point such that  $ \cos s_*(x,y) = \inp xy.$  
As mentioned before, $S_c(x,y)$ converges to $s_*(x,y)$ as $\cD(x,y)\to 0$.
Indeed, note that $\mathcal R(x,y,\tau) = (\tau - \inp xy)(\inp xy \tau - 1) + \cD(x,y)\tau$. Since 
  $ \mathcal R(x,y,\cos S_c)=0$,  this gives 
\begin{align*}
    \cos S_c - \inp xy =  \mathcal D(x,y)\cos S_c\, (1-\inp xy \cos S_c )^{-1}.
\end{align*}
Hereafter, we occasionally drop the variables $x,y$ to simplify the notation as long as no ambiguity arises. 
Since $1-\cos S_c\sim \mu$  for $(x,y)\in \tilde Q\times \tilde Q'$, by \eqref{inp} we have $1-\inp xy \cos S_c \gtrsim \mu$.
Hence,  it follows that  
\Be\label{distscxy}
|\cos S_c - \inp xy|\lesssim |\cD(x,y)|\mu^{-1} \ll \vepc \mu
\Ee
for $(x,y)\in \tilde Q\times \tilde Q'$.  From \eqref{inp}  we see  $s_\ast(x,y) \sim \mu^{1/2}$. Thus,  it follows that $|S_c(x,y) - s_*(x,y)|\lesssim |\cD(x,y)|\mu^{-3/2}$ for 
$(x,y)\in \tilde Q\times \tilde Q'$.

\subsubsection*{Decomposition away from $S_c$} 
We now break  
\[
  \tilde\chi_Q\fP_\lambda[\psic\psi_j]\tilde\chi_{Q'}= \sum_l\fP^*_l := \sum_l \tilde\chi_Q\fP_\lambda[\psi_j\psic{\psi}(2^l|\cdot-S_c|)]\tilde\chi_{Q'}.
\]
Note that $\fP^*_l \neq 0$ only if $2^{-l}\lesssim (\vepc\mu)^{1/2}$. 
To handle $\fP^*_l$, changing variables 
\[s\to S_c^l(x,y,s) := 2^{-l}s+S_c(x,y),\] 
we write
\Be\label{fpl}
\fP^*_l(x,y) = \mu^{-\frac d4}2^{-l}\int e^{i\lambda\Phi(x,y,s)}A(x,y,s) ds , 
\Ee
where
\begin{align*}
    &\Phi(x,y,s) = \phi_1(x,y,S_c^l(x,y,s)),\\
    &A(x,y,s) = \tilde\chi_Q(x)\tilde\chi_{Q'}(y)(\mu^{\frac d4}\mathfrak a\psic\psi_j)(S_c^l(x,y,s)){\psi}(|s|).
\end{align*}

Note that $\partial_s^2\Phi = 2^{-2l}(\partial_s^2\phi_1)(x,y,S_c^l)$. Since $\inp xy\sim 1$ and ${\sin S_c^l}\sim \mu^{1/2}$ on $\supp A$,
  \eqref{ph-dd} and  $\mathcal R(x,y,\cos s)=\inp xy (\cos S_c - \cos s) (\tau^+-\cos s)$ give
\[
|\partial_s^2\phi_1(x,y,S_c^l)|\sim   \mu^{-\frac32} {|\cos S_c - \cos S_c^l||\tau^+ - \cos S_c^l|}. 
\]
Note that $1- \cos S_c^l\sim \mu$ and $|\cos S_c^l - \cos S_c|\sim 2^{-l}\mu^{1/2}$ on $\supp  A$. 
Since 
$
\tau^+(x,y) - 1 = {(|x+y| + |x-y|)|x-y|}/{2\inp xy}\sim \mu
$
(see \eqref{x-y}), 
we also have  $ \tau^+-\cos S_c^l\sim \mu$ on  $\supp  A$.  Thus, we have
\Be\label{2ndbd}
|\partial_s^2\phi_1(x,y,S_c^l)|\sim  2^{-l} 
\Ee
on $\supp A$. Note $\partial_s^n A(x,y,s) = O(1)$ for $n\in \N_0$. By van der Corput's lemma  we get 
\Be \|\fP^*_l\|_{1\to \infty}\lesssim \lambda^{-\frac12}\mu^{-\frac d4}2^{\frac l2}. \label{goal3_1}\Ee

We also have the following estimates:

\begin{prop} Let $2^{-l}\lesssim (\vepc\mu)^{1/2}$. Then, the following estimates hold: 
\begin{align}
     &\|\fP^*_l\|_{2\to \infty}\lesssim \lambda^{-\frac12}\mu^{\frac{d-2}{4}}2^{-\frac l2}, \label{goal3_2}\\
    &\|\fP^*_l\|_{2\to 2}\lesssim \lambda^{-\frac d2}2^{-l}.\label{goal3_3}
\end{align}
\end{prop}
Using these estimates, one can easily verify the desired estimates \eqref{go1} and \eqref{go2}. Indeed, 
applying Lemma \ref{s-trick} to the estimates \eqref{goal3_1} and \eqref{goal3_3}, we get the restricted weak type estimate \eqref{go1}.
The restricted type estimate in \eqref{go2}  can be obtained similarly using  \eqref{goal3_1} and \eqref{goal3_2} when $d=2$. If $d\ge3$, interpolating \eqref{goal3_1} and \eqref{goal3_2} and then taking sum over $l$ give the desired estimate, a strong type estimate for $(p,q)=(\frac{2d}{d+1},\infty)$.

To complete the proof, it remains to prove the estimates \eqref{goal3_2} and \eqref{goal3_3}. 

\subsection{Proof of \eqref{goal3_2}} We begin with recalling the following bounds on derivatives of $\Phi$ and $A$, which was proved in \cite{JLR_endpoint}.
In fact,  the estimate \eqref{Aest} below was shown only  for $\beta=0$ in \cite[Lemma 4.4]{JLR_endpoint}. However, one can easily show \eqref{Aest}  using \eqref{partialest} 
and following the argument there.

\begin{lem}\cite[Lemma 4.4, 4.5]{JLR_endpoint}
Let $2^{-l}\le (\vepc\mu)^{\frac12}$ and $s\in \supp \psi(|\cdot|)$. If $(x,y)\in \tilde Q\times \tilde Q'$, then for any $\alpha,\beta\in \N_0^d$
\begin{align}\label{partialest}
    |\partial_x^\alpha\partial_y^\beta S_c(x,y,s)|&\lesssim \mu^{\frac12-|\alpha|-|\beta|}, 
    \\
    \label{Aest} |  \partial_x^\alpha\partial_y^\beta  A(x,y,s)| &\lesssim \mu^{-|\alpha|-|\beta|},
     \\
    \label{phiest} |\partial_x^\alpha\partial_y^\beta \Phi(x,y,s)|&\lesssim \mu^{\frac32-|\alpha|-|\beta|}.
\end{align}
\end{lem}

By \eqref{2-infty}, to prove \eqref{goal3_2} it is sufficient to show 
\Be  
\label{Pi4} \int  |\fP^\ast_l(x,y)|^2 dy\lesssim 2^{-l}\lambda^{-1} \mu^{\frac{d-2}{2}}.
\Ee
We write   
\Be
\label{Pi3}
\int  |\fP^\ast_l(x,y)|^2 dy= 2^{-2l} \mu^{-\frac d2}\iint\bigg(\int  \bar A(x,y,s,t) e^{i\lambda\Psi(x,y,s,t)}dy\bigg) dtds,
\Ee
where 
\begin{align}
 \Psi(x,y,s,t)&=\cP(x,y,\sxyls )-\cP(x,y,\sxylt)\\
\bar A(x,y,s,t)&=  A(x,y,s) \overline{A(x,y,t)}
\end{align} 

We now claim that 
\Be 
\label{lower-deriv}
|\partial_y \Psi(x,y,s,t)|\gtrsim 2^{-l} |s-t|, \quad \forall (x,y)\in  \tilde Q\times \tilde Q'
\Ee
holds if we take  $\vepc$ sufficiently small.  
From \eqref{Aest} and \eqref{phiest}, it follows that $\partial_y^\alpha \bar A=O(\mu^{-|\alpha|})$ and $
\partial_y^\alpha \Psi=O(\mu^{-|\alpha|+1}2^{-l} |s-t| )
$
for $(x,y)\in \tilde Q\times \tilde Q'$.  
Thus, by using \eqref{lower-deriv},  routine integration by parts  gives \footnote{One may rescale, that is to say, $y\to \mu y+ \mathbf c_{Q'}$ where $\mathbf c_{Q'}$ is the center of $Q'$. } 
\[
\Big| \int  \bar A(x,y,s,t) e^{i\lambda\Psi(x,y,s,t)}dy \Big| \lesssim  \mu^d (1+\lambda \mu 2^{-l}|s-t|)^{-N} . 
\]
Combining this with \eqref{Pi3} and integrating in $s$,  we get \eqref{Pi4}.

It remains to show \eqref{lower-deriv}.  We write  
\[\partial_y \Psi(x,y,s,t)=E+F,\] where 
\begin{align*}
E&=
\big(\partial_s \cP( x,y,\sxyls)-\partial_s \cP(x,y,\sxylt)\big)\partial_y S_c(x,y),
\\
F&=
\partial_y \cP(x,y,\sxyls)-\partial_y \cP(x,y,\sxylt).
\end{align*}
The mean value theorem gives
$E=\partial_s^2 \cP( x,y, S_c^l(x,y,s^*))2^{-l}(s-t)\partial_y S_c $
for some $s^*\in (s,t)$.  Since $2^{-l}\lesssim (\vepc\mu)^{1/2}$, using \eqref{2ndbd} and  \eqref{partialest}, we see $E = O(2^{-2l}\mu^{-1/2}|s-t|) = O(\vepc^{1/2}2^{-l}|s-t|)$.
Therefore, to show \eqref{lower-deriv}  it is enough to verify 
\Be\label{eq:F} |F|\gtrsim 2^{-l} |s-t| \Ee taking $\vepc>0$ small enough. 
To show \eqref{eq:F}, we exploit the  form of $\cP$.   
For simplicity,  fixing $x,y$,  we denote  $S^l(s)=2^{-l}s+S_c(x,y)$.  
By a direct computation we get
\[
F  
=\mathbf a y - \mathbf b (x-y).
\] 
where 
\[\mathbf a= \Big(\frac{\cos S^l(s) -1 }{\sin S^l(s) } - \frac{\cos S^l(t)-1}{\sin S^l(t) }\Big), \qquad \mathbf b= \Big(\frac{1 }{\sin S^l(s) } -\frac1{\sin S^l(t) }\Big). \] 
By the mean value theorem  it is clear that $|\mathbf a|\sim  2^{-l} |s-t| $ and, similarly, $|\mathbf b|\sim   2^{-l} \mu^{-1} |s-t|$.
Since $|x-y|\sim \mu$,  \eqref{eq:F} follows once we show 
\Be
 \label{angleab}
\sin^2 \theta(y,x-y)\sim 1, \quad \forall (x,y)\in \tilde Q\times \tilde Q'.  
\Ee
Now, we recall \eqref{angle2}, so  \eqref{angleab} follows from \eqref{x-y} since  $
|y|^2|x-y|^2\sin^2 \theta(y,x-y) = |x|^2|y|^2\sin^2 \theta(x,y).$  
\qed

\subsection{Proof of \eqref{goal3_3}}
Let us define an oscillatory integral operator $I_s^\lambda(\Phi,A)$ by
\[
I_s^\lambda(\Phi,A) f(x) = \int e^{i\lambda\Phi(x,y,s)} A(x,y,s) f(y) dy.
\]
We observe that  $\fP^*_l f = \mu^{-\frac d4}2^{-l}\int I_s^\lambda(\Phi,A) f ds$. By Minkowski's inequality we see
 $\|\fP_l^*\|_{2\to 2}\le \mu^{-\frac d4}2^{-l}\sup_s \|I_s^\lambda(\Phi,A)\|_{2\to 2}$.
Thus,  \eqref{goal3_3} follows from
\Be\label{goal4}
\|I_s^\lambda(\Phi,A) f\|_2\lesssim \lambda^{-\frac d2}\mu^{\frac d4}\|f\|_2,\quad \forall s.
\Ee
Note that $I_s^\lambda(\Phi,A)\neq 0$ only if  $s\in \supp  \wt{\psi}$. 
To obtain \eqref{goal4}, we make use of the following well known lemma.

\begin{lem}[{\cite[p. 377]{St93}}]
\label{nonzero-det}  Let  $a\in C^\infty_c(\R^d\times \R^d)$ and $\phi$ be a smooth function on $\supp a$. 
Let 
\[
T_R f(x) := \int e^{i R\phi(x,y)} a(x,y) f(y) dy, \quad R>0.
\]
Suppose that $\det{\partial_x\partial_y^\intercal\phi} \neq 0$ on $\supp{a}$.  
Then, $
\|T_R \|_{2\to 2}\le CR^{-\frac d2}
$
for any $R>0$ and  the constant $C$ is stable under small smooth perturbation of $\phi$ and $a$. 
\end{lem}

To apply the lemma, we  verify that the mixed Hessian of $\Phi(\cdot,\cdot,s)$ is invertible for any $s$, that is, $\det{\partial_x\partial_y^\intercal\Phi(x,y,s)} \neq 0$.

\begin{lem}\label{detPhi}
Let $\lambda^{-\frac 23}\le\mu\le 1/4$ and $2^{-l}\le (\vepc\mu)^{1/2}$. If $\vepc>0$ is sufficiently small, then $\det{\partial_x\partial_y^\intercal\Phi(x,y,s)} \sim \mu^{-d/2}$ for any $(x,y,s)\in \tilde Q\times \tilde Q'\times\supp \psi(|\cdot|)$.
\end{lem}

 Assuming this for the moment, we prove \eqref{goal4}.

\begin{proof}[Proof of \eqref{goal4}]
Recall that $\mathbf c_Q$ denotes the center of a cube $Q$. 
We denote $   l_\mu(x,y) = (\mu x,\mu y) + (\mathbf c_Q, \mathbf c_{Q'})$ and 
 set
\begin{align*}
    \tilde{\Phi}(x,y,s) &= \mu^{-\frac32}\Phi(l_\mu(x,y),s),
    \\
     \tilde{A}(x,y,s) & = A(l_\mu(x,y),s). 
\end{align*}
Then, changing variables $(x,y)\to l_\mu(x,y)$, we have
\[
\|I_s^\lambda(\Phi,A)\|_{2\to 2} = \mu^{d}\|I_s^{\lambda \mu^{3/2}} (\tilde{\Phi},\tilde{A})\|_{2\to 2}.
\]

We apply Lemma \ref{nonzero-det}  to  $I_s^{\lambda \mu^{3/2}} (\tilde{\Phi},\tilde{A})$.  Using \eqref{Aest} and \eqref{phiest}, one can easily see
\[
|\partial_x^\alpha\partial_y^\beta \tilde{\Phi}(x,y,s)|\lesssim 1, \quad |\partial_x^\alpha\partial_y^\beta \tilde{A}(x,y,s)|\lesssim 1, \quad \forall (x,y,s)\in \supp \tilde A.
\]
By Lemma \ref{detPhi}  it follows that $\det \partial_x\partial_y^\intercal \tilde \Phi\sim 1$. 
Applying Lemma \ref{nonzero-det} gives  the estimate $\| I_s^{\lambda \mu^{3/2}} (\tilde{\Phi},\tilde{A})\|_{2\to 2}\lesssim \lambda^{-d/2}\mu^{- 3d/4}$.
Therefore, we  get \eqref{goal4} as desired.
\end{proof}

\begin{proof}[Proof of Lemma \ref{detPhi}]  
We begin with claiming that
\Be
\label{hho}
\partial_x\partial_y^{\intercal}   \Phi(x,y,s) = \mathbf H +O( \vepc^{1/2}\mu^{-1/2}) 
\Ee
for  $(x,y,s)\in \tilde Q\times \tilde Q'\times\supp \psi(|\cdot|)$ where 
\begin{equation}
\label{hhh}
\mathbf H = -(\sin S_c)^{-1}\mathbf I +  
\partial_x S_c \partial_y^{\intercal}  \partial_s \phi_1(x,y, S_c)  + \partial_x\partial_s\phi_1(x,y,S_c) 
 \partial_y^{\intercal}  S_c.
 \end{equation}
 Here $\mathbf I$ denotes the $d\times d$ identity matrix.   Using the chain rule, we write
\begin{equation}
\label{partialxy}
\begin{aligned}
\partial_x\partial_y^{\intercal} &  \Phi(x,y,s) =  \partial_x \partial_y^{\intercal} \phi_1(x,y, \sxyl)
+
\partial_x S_c \partial_y^{\intercal}  \partial_s \phi_1(x,y,  \sxyl) 
\, + 
\\[3pt]
\quad \partial_x&\partial_s\phi_1(x,y,  \sxyl) \partial_y^{\intercal}  S_c 
 +\partial_s^2\phi_1(x,y,  \sxyl) 
\partial_x S_c  \partial_y^{\intercal}  S_c +  \partial_s\phi_1(x,y,  \sxyl) 
 \partial_x\partial_y^{\intercal}  S_c.
\end{aligned}
\end{equation}

 By \eqref{2ndbd} and \eqref{partialest}, we see  $\partial_s^2\phi_1(x,y,  \sxyl)\partial_x S_c  \partial_y^{\intercal}  S_c =
O(2^{-l}\mu^{-1}) = O( \vepc^{1/2}\mu^{-1/2})$. Since   $2^{-l}\lesssim (\vepc\mu)^{1/2}$ and $|\cos S_c^l - \cos S_c|\sim 2^{-l}\mu^{1/2}$,  by \eqref{distscxy} 
it follows that $|\cos \sxyl -\inp xy|=O(\vepc^{1/2}\mu)$.  Thus, by \eqref{ph-d} we see $\partial_s\phi_1(x,y,S_c^l) = O(\vepc\mu)$ because  $|\cD|\ll\vepc\mu^2$. 
Combining this and \eqref{partialest}, we have $\partial_s\phi_1(x,y,  \sxyl) 
 \partial_x\partial_y^{\intercal}  S_c=O(\vepc\mu^{-1/2})$.  Therefore, we need only to consider the other terms in the right hand side of \eqref{partialxy}. 
  Therefore, to show \eqref{hho}, we note  $\partial_x\partial_y^\intercal\phi_1(x,y,S_c) = -(\sin S_c)^{-1}\mathbf I$, and
\Be\label{parxy}
\partial_x\partial_s\phi_1(x,y,S_c) = \frac{y\cos S_c - x}{\sin^2 S_c},\quad \partial_y\partial_s\phi_1(x,y,S_c) = \frac{x\cos S_c - y}{\sin^2 S_c}.
\Ee
Since $|\cos S_c^l - \cos S_c|\sim 2^{-l}\mu^{1/2}$ and $|\sin S_c^l - \sin S_c|\sim 2^{-l}$, from the above identities we see that $S_c^l$ appearing in the first to third terms on the right hand side of \eqref{partialxy} can be replaced by $S_c$  allowing an error of  $O(\varepsilon_1\mu^{-1/2})$. Thus, we get \eqref{hho}.

Thanks to \eqref{hho},  the matter reduces to showing that $\det\mathbf H \sim \mu^{-d/2}$. To this end, we need to obtain precise expressions for $\partial_x S_c$, $\partial_y^\intercal S_c$. 
We set
\Be\label{vw}
v(x,y)=y\cos S_c - x,\quad  w(x,y)=x\cos S_c - y. 
\Ee
From  \eqref{sc}, \eqref{x-y}, and \eqref{angleab},   we note that   the two vectors $(\cos S_c - 1)x$ and $x-y$ are of size $\sim \mu$ and are separated by $\sim \mu$.  Thus, writing $ w = (\cos S_c - 1)x + (x-y)$, we see  $|w|\sim \mu$. In the same manner, it follows that  $|v|\sim \mu$. So, we have 
\Be 
\label{size-vw} 
|v(x,y)|\sim \mu,\quad   |w(x,y)|\sim \mu. 
\Ee 

We now observe that  $v^\intercal w = \mathcal R(x,y,  \cos S_c)$. This gives
\Be\label{orthogonal}
 v^\intercal w= 0.
\Ee
Differentiating in $x$, we have $ \partial_x(v^\intercal w) = -\sin S_c \partial_x S_c(x^\intercal v+y^\intercal w) + \cos S_c v - w$. Combining this and \eqref{orthogonal} yields  
\begin{align*}
    \partial_x S_c = -(\sin S_c)^{-1}(x^\intercal v + y^\intercal w)^{-1} (w - \cos S_c v).
\end{align*}
Similarly, we have $\partial_y^\intercal S_c = -(\sin S_c)^{-1}(x^\intercal v+y^\intercal w)^{-1}(v^\intercal-\cos S_c w^\intercal)$, which can also be shown using 
symmetry of $S_c$ and $v(y,x)=w(x,y)$.  From \eqref{vw}, we also note that $\partial_x\partial_s\phi_1(x,y,S_c) = (\sin S_c)^{-2} v$ and $\partial_y\partial_s\phi_1(x,y,S_c) = (\sin S_c)^{-2} w$.
Applying these identities to \eqref{hhh}, we obtain 
\[
\mathbf H = -\frac{\sin^2 S_c (x^\intercal v+y^\intercal w)\mathbf I + vv^\intercal + ww^\intercal -2\cos S_c vw^\intercal}{\sin^3 S_c \,(x^\intercal v+y^\intercal w)}.
\]
By \eqref{orthogonal} and \eqref{vw}, it follows that $
x^\intercal v\cos S_c = y^\intercal v$ and $y^\intercal w\cos S_c = x^\intercal w$. Using those and $\sin^2 S_c=1-\cos^2 S_c$, we see
\begin{align*}
    \sin^2 S_c(x^\intercal v+y^\intercal w) & = x^\intercal v+y^\intercal w - \cos S_c (y^\intercal v+x^\intercal w) = -|v|^2-|w|^2.
\end{align*}
For the last equality we use \eqref{vw} again.  Thus, we obtain 
 \[ \mathbf H = -\frac1{\sin S_c}\,\mathbf L,\]
  where
\[
\mathbf L =\mathbf I - \frac1{ (|v|^2+|w|^2)} (vv^\intercal + ww^\intercal- 2\cos S_c vw^\intercal).
\]
 To complete the proof, it remains to show that $\det\mathbf L \sim 1$
since $\sin S_c\sim \mu^{1/2}$. 

Let $\mathrm V=\text{span}(v,w)$. Then, it is easy to see that $\mathrm V$ is an invariant subspace of $\mathbf L$.    Let $\mathbf M$ denote the matrix of the  linear map $\mathbf L|_{\mathrm V}$ (restricted to $\mathrm V$) with respect to bases  $v, w$.   
  Using \eqref{orthogonal},  we see  $\mathbf Lv=  (|v|^2+|w|^2)^{-1} |w|^2 v$ and $\mathbf Lw= (|v|^2+|w|^2)^{-1} (  |w|^2 2\cos S_c  v +  |v|^2 w).$ 
  Thus, we obtain 
  \[
\mathbf M= \frac{1}{|v|^2+|w|^2} 
\begin{pmatrix}
\,|w|^2 &  2|w|^2 \cos S_c 
\\
0 & |v|^2
\end{pmatrix}.
\]
Note  that $\rm V^\perp$ is also an invariant subspace and $\mathbf L u=u$ for all $u\in \mathrm V^\perp$. So,  $\det \mathbf L=\det \mathbf M$ and $\det \mathbf L=  (|v|^2+|w|^2)^{-2}  |v|^2|w|^2$. 
By \eqref{size-vw}  we conclude $\det\mathbf L \sim 1$. 
\end{proof}

\begin{rem}\label{frmk}
When $\mu\sim 1$,  to prove \eqref{weak2} and \eqref{scaled5/6}  is much simpler. Especially, we don't need the sectorial decomposition. One may just cover $A_\mu\times A_\mu$ with disjoint cubes $\{Q\times Q'\}$  of small side length $c\vepc$ so that either   \eqref{large-d} or \eqref{small-det}
holds. We only need to consider the latter case since the first can be handled easily  as before.  Since \eqref{small-det} holds and $
\mu\le 1/4$, taking $c,\vepc$ small enough, we have 
\[ |\inp xy|\ge c_0, \quad  (x,y)\in Q\times Q'\] 
for  a constant $c_0>0$.   So, the estimates for $\chi_Q \fP_\lambda[\psi_j]\chi_{Q'}$ are easy to show since 
$\mathcal Q(x,y, \cos s)\sim 1$ for $(x,y,s)\in Q\times Q'\times \supp \psi_j$ when $2^{-j}\ll 1$. Thus, we may assume $2^{-j}\sim 1$.  
We now note that \eqref{angle2} and \eqref{x-y} holds with $\mu\sim 1$.  Thus, the previous proofs of \eqref{goal3_2} and \eqref{goal3_3}  work without modification. 
\end{rem}

\section{ Lower bounds on $\|\chi_{\lambda,\mu} \Pi_\lambda\chi_{\lambda,\mu}\|_{p\to q}$ 
}\label{sub_counter}
 In this section, we show that  the bound \eqref{est-ann} can not be improved, that is to say, there is a constant $C>0$ such that 
\Be \label{lowerbd}  \|\chi_{\lambda,\mu} \Pi_\lambda\chi_{\lambda,\mu}\|_{p\to q}\ge C\lambda^{\beta(p,q)}\mu^{\gamma(p,q)}. \Ee
For the purpose,  it is sufficient to show the  lower bounds \eqref{ann1}--\eqref{ann3} in Proposition \ref{lower-mu} below. 
Comparing those lower bounds immediately  yields  \eqref{lowerbd}. 

\begin{prop}
\label{lower-mu}
Let $d\ge 2$, $\lambda\gg 1$, and $\lambda^{-2/3}\le\mu\le 2^{-1}$. Then, we have 
\begin{align}\label{ann1}
    &\|\chi_{\lambda,\mu} \Pi_\lambda \chi_{\lambda,\mu} \|_{p\to q}\gtrsim(\lambda\mu)^{-1+\frac d2(\frac1p-\frac1q)}, 
   \\
    \label{ann2}
    &\|\chi_{\lambda,\mu} \Pi_\lambda \chi_{\lambda,\mu} \|_{p\to q}\gtrsim\lambda^{-\frac12(\frac1p-\frac1q)}\mu^{\frac12-\frac{d+3}{4}(\frac1p-\frac1q)}, 
  \\
    \label{ann3}
    &\|\chi_{\lambda,\mu} \Pi_\lambda \chi_{\lambda,\mu} \|_{p\to q}\gtrsim\lambda^{-\frac12+\frac d2(1-\frac1p-\frac1q)}\mu^{-\frac14+d(\frac34-\frac1p-\frac{1}{2q})}.
\end{align}
\end{prop}

In particular,   the lower bounds in Proposition \ref{lower-mu} with $\mu=2^{-1}$   also shows  sharpness of the estimate \eqref{est-loc0}.

The  lower bounds $\eqref{ann1}$ and $\eqref{ann2}$ which yield \eqref{lowerbd} for $\ppq\in \mathcal{R}_1\cup \overline{\mathcal{R}}_3$ can be shown 
by using the known lower bound on $\|\chi_{\lambda,\mu} \Pi_\lambda \chi_{\lambda,\mu} \|_{q'\to q}$. We recall \eqref{kt-nu} of which sharpness was shown in \cite{T05}. By the  $TT^*$ argument we have 
\begin{align}\label{est:ann1,2}
\|\chi_{\lambda,\mu} \Pi_\lambda \chi_{\lambda,\mu} \|_{p\to p'}\ge C\lambda^{\beta(p,p')}\mu^{\gamma(p,p')}
\end{align}
 for $1\le p\le 2$ with  $C>0$ depending only  on $d$. Suppose that  \eqref{lowerbd} fails for some $(1/p_0, 1/q_0)\in \mathcal{R}_1\cup \overline{\mathcal{R}}_3$, $p_0'\neq q_0$,  that is to say,  there are sequences $\lambda_k, \mu_k$ such that 
\[  C_k:= \lambda_k^{-\beta(p_0,q_0)}\mu_k^{-\gamma(p_0,q_0)}{\|\chi_{\lambda_k,\mu_k} \Pi_\lambda \chi_{\lambda_k,\mu_k} \|_{p_0\to q_0}}\to 0\] 
as $k\to \infty$. 
By duality  we have $ \|\chi_{\lambda_k,\mu_k} \Pi_\lambda \chi_{\lambda_k,\mu_k} \|_{q'_0\to p'_0}\le C_k\lambda_k^{\beta(p_0,q_0)}\mu_k^{\gamma(p_0, q_0)}$.
Since $p_0'\neq q_0$,  interpolation with the estimate  ${\|\chi_{\lambda_k,\mu_k} \Pi_\lambda \chi_{\lambda_k,\mu_k} \|_{p_0\to q_0}}\le C_k \lambda_k^{\beta(p_0,q_0)}\mu_k^{\gamma(p_0,q_0)}$  
gives  $\|\chi_{\lambda_k,\mu_k} \Pi_{\lambda_k} \chi_{\lambda_k,\mu_k} \|_{p_\ast \to p_\ast'}\le C_k \lambda_k^{\beta(p_\ast ,p_\ast')}\mu_k^{\gamma(p_\ast ,p_\ast')}$ where $1/p_\ast-1/p_\ast'=1/p_0-1/q_0$ because $\beta(p_0, q_0)=\beta(p_\ast ,p_\ast')$ and $\gamma(p_0, q_0)=\gamma(p_\ast ,p_\ast')$. 
This contradicts to the lower bound  \eqref{est:ann1,2} if we let $k\to \infty$.

\subsection*{Proof of \eqref{ann3}} 
It remains to prove  \eqref{ann3}.  To do so, we mainly rely on  asymptotic properties of the Hermite functions. Let $h_k(t)$ denote the $L^2$-normalized $k$-th Hermite function  of which eigenvalue is  $2k+1$. We make use of the following lemma from \cite{T05}.   Also see \cite{AAR} and \cite{F93}.

\begin{lem}\cite[Lemma 5.1]{T05}\label{S4-asym} Let $\nu=\sqrt{2k+1}$. We set 
\[
    s^-_\nu(t)=\int_0^t\sqrt{|\tau^2-\nu^2|}d\tau\quad \text{and}\quad s^+_\nu(t)=\int^t_\nu\sqrt{|\tau^2-\nu^2|}d\tau. 
\]
Then,  the following hold\,$: $
\[  \  \  \  h_{2k}(t)
                =\begin{cases}
                       a_{2k}^-(\nu^2-t^2)^{-\frac14} \big(\!\cos s^-_\nu(t)+\mathcal E\big),     & \ \qquad\qquad\quad \ |t|<\nu-\nu^{-\frac13}, 
                        \\
                    \qquad \qquad O(\nu^{-\frac16}),                                           & \ \ \, \nu-\nu^{-\frac13}<|t|<\nu+\nu^{-\frac13}, 
                     \\
                      a^+_{2k}e^{-s^+_\nu(|t|)}(t^2-\nu^2)^{-\frac14}\big(\!1+\mathcal E\big),     &\ \,  \ \nu+\nu^{-\frac13}<  |t|,
                \end{cases}
\]
\vspace{4pt}
\[    h_{2k+1}(t)
              =\begin{cases}
               a_{2k+1}^-(\nu^2-t^2)^{-\frac14}\big(\! \sin\,s^-_\nu(t)+\,\mathcal E\big),   & \qquad\qquad \  \, \ |t|<\nu-\nu^{-\frac13},  
               \\
              \qquad \qquad  O(\nu^{-\frac16}),      & \nu-\nu^{-\frac13}<|t|<\nu+\nu^{-\frac13},  
              \\
             a^+_{2k+1}e^{-s^+_\nu(|t|)}(t^2-\nu^2)^{-\frac14}\big(\!1+\mathcal E\big),    &\nu+\nu^{-\frac13}<|t|,
             \end{cases}
             \]
where
$    |a^{\pm}_k|\sim 1$  and $ \mathcal E=O\big(|t^2-\nu^2|^{-\frac12}\left||t|-\nu\right|^{-1}\big) .$
\end{lem}

We also need the following lemma. 

\begin{lem}\label{lem:annn3}  Let $1\le p\le 2$ and  $x_0\in A_{\lambda,\mu}$. Then, if $\lambda^{-2/3}\lesssim \mu\le1/2$,  we have
\[
\|\Pi_\lambda(x_0,\cdot)\|_{L^p(A_{\lambda,\mu})}\lesssim
\lambda^{-\frac12+\frac{d}{2p}}\mu^{-\frac14+d(\frac1p-\frac14)}.
\]
\end{lem}

\begin{proof}
If $\mu\sim 1$, the estimate follows by H\"older's inequality and the bound $\|\Pi_\lambda\|_{1\to \infty}\lesssim \lambda^{(d-2)/2}$.
Indeed, since $|A_{\lambda,\mu}|\sim \lambda^{d/2}$ and $\|\Pi_\lambda(x_0, \cdot)\|_2^2=\Pi_\lambda(x_0,x_0)$, we see
\[
\|\Pi_\lambda(x_0,\cdot)\|_{L^p(A_{\lambda,\mu})}
\le  
|A_{\lambda,\mu}|^{\frac1p-\frac12}|\Pi_\lambda(x_0,x_0)|^{\frac12}\lesssim \lambda^{-\frac12+\frac{d}{2p}}.
\]

Thus, we may assume $\mu\ll 1$. Let $S_{0}=  \big( B(x_0,  C \lambda^{1/2}\mu)\cup B(-x_0,  C \lambda^{1/2}\mu) \big)\cap  A_{\lambda,\mu}$ for a large constant $C>0$. 
We note $\|\Pi_\lambda(x_0,\cdot)\|_{L^2(S_0)}\le \|\chi_{\lambda,\mu} \Pi_\lambda\chi_{\lambda,\mu}\|_{2\to \infty}$ from \eqref{2-infty}. Recalling  
 \eqref{est-ann} for $p=2$ and $q=\infty$, we have $\|\Pi_\lambda(x_0,\cdot)\|_{L^2(S_0)}\le  C\lambda^{(d-2)/4}\mu^{(d-1)/4}$.  Thus, by H\"older's inequality  we have 
\[ \|\Pi_\lambda(x_0,\cdot)\|_{L^p(S_0)}  \le |S_0|^{\frac1p-\frac12} \|\Pi_\lambda(x_0,\cdot)\|_{L^2(S_0)} \lesssim    \lambda^{-\frac12+\frac{d}{2p}}\mu^{-\frac14+d(\frac1p-\frac14)}.  \]
So, it is sufficient to show 
\[ \|\Pi_\lambda(x_0,\cdot)\|_{L^p(A_{\lambda, \mu}\setminus S_0)}  \lesssim   \lambda^{-\frac12+\frac{d}{2p}}\mu^{-\frac14+d(\frac1p-\frac14)}.    \]
By the symmetric property of the kernels of $\Pi_\lambda[\psi_j]$, $\Pi_\lambda[\psi_j^\kappa]$, $\kappa=-, \pm \pi$ (\cite[p. 5]{JLR_endpoint}), it follows that 
$\|\sum_{j\ge 0}  \Pi_\lambda[\psi_j] (x_0,\cdot)\|_{L^p(A_{\lambda, \mu}\setminus S_0)}=\|\sum_{j\ge 0}  \Pi_\lambda[\psi_j^\kappa] (x_0,\cdot)\|_{L^p(A_{\lambda, \mu}\setminus S_0)}$, $\kappa=-, \pm \pi$. 
For the desired estimate,  by \eqref{k-decom} we need only to show that 
\Be \label{fin} \tx \| \sum_{j\ge 0}  \Pi_\lambda[\psi_j] (x_0,\cdot)\|_{L^p(A_{\lambda,\mu}\setminus S_0)}\lesssim
\lambda^{-\frac12+\frac{d}{2p}}\mu^{-\frac14+d(\frac1p-\frac14)}.\Ee

For $ l\ge 1$, set $S_l =\{x\in A_{\lambda,\mu} : |x-x_0|\in  C\lambda^{\frac12}\mu [ 2^{l-1}, 2^{l})\},$ 
To prove the above estimate,   it is enough to show 
\begin{align*}
  \tx  \|\sum_{j\ge 0}  \Pi_\lambda[\psi_j] (x_0,\cdot)\|_{L^p(S_l)}
    \lesssim
    (\mu 2^l) ^{-\frac{d-2}{4}}
    (\lambda^{\frac d2}\mu^d 2^{dl})^\frac1p (\lambda(\mu 2^l)^{\frac32})^{-N},
    \quad l \ge 1.
\end{align*}
Summation over $l$ gives the estimate \eqref{fin}  because $\lambda\mu^{3/2}\gtrsim 1$.  Using  \eqref{s-decom} and scaling,  we observe 
\[\tx\|\sum_{j\ge 0}  \Pi_\lambda[\psi_j] (x_0,\cdot)\|_{L^p(S_{l})}\lesssim \sum_{j\ge 0} \sup_{k\sim_{\nu}k'}\|\chi_k^{\nu}\mathfrak P_\lambda[\psi_j] \chi_{k'}^{\nu}\|_{L_{x,y}^\infty}\,|S_{l}|^{1/p},\]
where  $\nu$ satisfies $2^{-\nu}\sim C\mu 2^{l}$.  Note that $\cD\sim -(\mu2^l)^2$ if $(x,y)\in \supp{\chi_k^{\nu}}\times \supp{\chi_{k'}^{\nu}}$, $k\sim_{\nu}k'$. 
Thus, using  $(a)$ in Lemma \ref{oscillatory},
 we obtain
\[\tx\sum_{j\ge 0} \sup_{k\sim_{\nu}k'}\|\chi_k^{\nu}\mathfrak P_\lambda[\psi_j] \chi_{k'}^{\nu}\|_{L_{x,y}^\infty}\lesssim (\mu 2^l)^{-\frac{d-2}{4}}(\lambda(\mu 2^l)^{\frac32})^{-N},\]
for any  $N\in\N$. Therefore, we get the desired estimate.
\end{proof}

To show \eqref{ann3}, we first claim that  there is a point  $x_0\in A_{\lambda,\mu}$ such that   
\Be
\label{eq:x0x0}
\int_{A_{\lambda,\mu}} \Pi_\lambda(x_0,y)^2 \,dy
\gtrsim  \mu^{\frac12}(\lambda\mu)^{\frac{d-2}{2}}.
\Ee
Combined with \eqref{2-infty}, this shows sharpness of the bound \eqref{est-ann} for $p=2$ and $q=\infty$. 
 Assuming \eqref{eq:x0x0} for the moment we prove \eqref{ann3}.  
Let us set 
\[f(x)=\Pi_\lambda(x_0,x)\chi_{A_{\lambda,\mu}}(x).\] 
By \eqref{eq:x0x0} we have $\Pi_\lambda f(x_0)\gtrsim \mu^{\frac12}(\lambda\mu)^{\frac{d-2}{2}}$.  
We now recall the following lemma.   

\begin{lem}\cite[Lemma 4.5]{LR}
    Let $\lambda\in 2\N_0+d$ and $\mu\in[\lambda^{-\frac23},\frac12]$. 
    Suppose that $h\in \mathcal S(\mathbb R^d)$ is an  eigenfunction of $H$ with eigenvalue $\lambda$, {\it i.e.,} $H h(x) = \lambda h(x)$. If  
    $
    y_0\in A_{\lambda,\mu}, 
    $
    then  for any $\alpha\in \N_0^d$ we have 
    \Be
    \label{eq:deriv} |\partial_y^\alpha h(y_0)|\le C (\lambda\mu)^{\frac{|\alpha|}{2}}\|h\|_{L^\infty(B(y_0,2(\lambda\mu)^{-\frac12}))}\Ee 
    with $C$ independent of $\lambda$, $\mu$ and $h$.
\end{lem}

By this lemma we also have 
$\|\nabla\,\Pi_\lambda f\|_{L^\infty(A_{\lambda,\mu})}\lesssim \mu^{1/2}(\lambda\mu)^{(d-1)/{2}}$. 
By the mean value theorem  we see that   $\Pi_\lambda f(x)\gtrsim  \mu^{1/2}(\lambda\mu)^{(d-2)/{2}}$ if $x\in B(x_0,c(\lambda\mu)^{-1/2})$ for a constant $c>0$ small enough.  Thus, 
we have
\[
\|\chi_{\lambda,\mu}\Pi_\lambda\chi_{\lambda,\mu} f\|_q\ge
\|\Pi_\lambda f\|_{L^q(B(x_0,c(\lambda\mu)^{-1/2})\cap A_{\lambda, \mu})}
\gtrsim \mu^{\frac12}(\lambda\mu)^{\frac{d-2}{2}-\frac{d}{2q}}.
\]
Combining this and the estimate $\|f\|_p\lesssim \lambda^{-\frac12+\frac{d}{2p}}\mu^{-\frac14+d(\frac1p-\frac14)}$, which follows from Lemma \ref{lem:annn3}, we obtain \eqref{ann3}.
 It remains to show \eqref{eq:x0x0}.  

\begin{proof}[Proof of  \eqref{eq:x0x0}] 
Let $\lambda=2N +d$. We set
$$ J=\big\{\alpha\in\mathbb N^d \,: \,|\alpha|=N,\; {N\mu}/(2^{10} d) \le\alpha_j\le {N\mu}/(2^9 d),  \ \  2\le j\le d\big\},$$
$\ell={(2\sqrt{d})^{-1}}{\sqrt{\lambda\mu}}$, and $Q_{\ell}=[ \sqrt{\lambda}(1-2\mu),\sqrt{\lambda}(1-3\mu/2)]\times  [-\ell,\ell\,]^{d-1}$.
Noting that $Q_{\ell}\subset A_{\lambda,\mu}$ and $|J|\sim (\lambda\mu)^{d-1}$,  we have
$$\sum_{\alpha\in J} \int_{B_{d-1}(0,(\lambda\mu)^{-\frac12})}\int_{\sqrt{\lambda}(1-2\mu)}^{\sqrt{\lambda}(1-3\mu/2)} \left|\Phi_\alpha (x_1,\bar x)\right|\, dx_1d{\bar x}\gtrsim \lambda^{\frac d4}\mu^{\frac{d+2}{4}}.$$
This is an easy consequence of Lemma \ref{S4-asym}. 
Thus, there exists $x_0\in [\sqrt{\lambda}(1-2\mu),\sqrt{\lambda}(1-3\mu/2)]\times B_{d-1}(0,(\lambda\mu)^{-\frac12})$ such that $\sum_{\alpha\in J}|\Phi_\alpha(x_0)|\gtrsim (\lambda\mu)^{\frac{3}{4}d-1}$.
We consider 
\[g(x)=\chi_{Q_{\ell}}(x)\sum_{\alpha\in J}\,c_\alpha \Phi_\alpha(x),\]
where $c_\alpha\in \{-1,1\}$ such that $c_\alpha\Phi_\alpha(x_0)=|\Phi_\alpha(x_0)|$.
By  \eqref{kt-nu} with $q=2$,  we have 
\[
\textstyle  \|g\|_2\le\|\sum_{\alpha\in J}c_\alpha \Phi_\alpha(x)\|_{L^2(A_{\lambda,\mu})}\lesssim\mu^{\frac14}\|\sum_{\alpha\in J}c_\alpha \Phi_\alpha(x)\|_2\lesssim \mu^{\frac14}(\lambda\mu)^{\frac{d-1}{2}}.
\]

We now set 
\begin{align*}
 a_{u,v}:=\int_{-\ell}^{\ell}h_u(t)h_v(t)dt, \quad a^\ast_{u,v}:=\int_{\sqrt{\lambda}(1-2\mu)}^{\sqrt{\lambda}(1-3\mu/2)}h_u(t)h_v(t)dt,
\end{align*} 
and $A_{\alpha,\beta}=\int_{Q_\ell}\Phi_{\alpha}(x)\Phi_{\beta}(x)dx$.
Note that $A_{\alpha,\beta} = a^\ast_{\alpha_{1},\beta_{1}}\prod_{i=2}^d a_{\alpha_{i},\beta_{i}}$
and  $\Pi_\lambda g(x)     = \sum_{\alpha\in J}  \sum_{\beta: |\beta|=N} c_\alpha A_{\alpha, \beta} \Phi_\beta(x)$. Thus, we  write
\Be
\label{g-lambda}
    \Pi_{\lambda}g(x) 
      ={\mathrm I} (x)+\mathrm {I\!I}(x),
  \Ee
where 
\[  \textstyle   {\mathrm I} (x)=  \sum_{\alpha\in J}c_{\alpha}A_{\alpha, \alpha}\Phi_{\alpha}(x), \quad  \mathrm {I\!I}(x)= \sum_{\alpha\in J} \sum_{\beta: |\beta|=N, \alpha\not= \beta} c_{\alpha} A_{\alpha, \beta} \Phi_{\beta}(x) .\]

Using Lemma~\ref{S4-asym}, it is easy to see that $a_{u,u}\sim 1$ and $a^\ast_{u,u}\sim \mu^{1/2}$ if $u\sim N$.  Consequently, it follows that 
$A_{\alpha, \alpha}\sim \mu^{1/2} \ \text{ if } \ \alpha\in J. $
However, if $u\neq v$, $a_{u,v}$  is exponentially decaying, thus we  may regard $ \mathrm {I\!I}$ as a minor error.  
More precisely,  
\Be\label{auv-final}
A_{\alpha, \beta}\lesssim e^{-c\lambda\mu}
\Ee 
if $\alpha\in J$, $|\beta|=N$, and $\alpha\neq \beta$.  Assuming this for the moment,  we  prove  \eqref{eq:x0x0}. 

By \eqref{auv-final} it follows that $\mathrm{I\!I}(x)=O\big((\lambda\mu)^a e^{-b\lambda\mu}\big)$ for some $a,b>0$.
Hence, our choices of $c_\alpha$ and $x_0$ ensures that there exists $C>0$ such that ${\mathrm I}(x_0)=\sum_{\alpha\in J}c_{\alpha}A_{\alpha, \alpha}\Phi_{\alpha}(x_0)\ge C\mu^{1/2}(\lambda\mu)^{3d/4-1} $. 
Since $Q_{\ell}\subset A_{\lambda, \mu}$ and $x_0\in A_{\lambda, \mu}$,  recalling \eqref{g-lambda}, we obtain 
\begin{align*}
\inp{\chi_{\lambda,\mu}(x_0)\Pi_\lambda(x_0,\cdot) \chi_{\lambda,\mu}}{g}=\Pi_\lambda g(x_0)\ge
C\mu^{\frac12}(\lambda\mu)^{\frac34 d-1}-O((\lambda\mu)^a e^{-b\lambda\mu}).
\end{align*}
Since $\|g\|_2\lesssim \mu^{1/4}(\lambda\mu)^{(d-1)/2}$, by duality we get \eqref{eq:x0x0} as desired.

We now show \eqref{auv-final}. Recalling the identity $2(u-v)h_uh_v =h_uh_v''-h_u''h_v$  (for example, see \cite[p. 2]{Th93}),  we have 
\[a_{u,v}
                          =\frac{1}{2(u-v)}\int_{-\ell}^{\ell} h_u(s)h^{''}_v(s)-h_u^{''}(s)h_v(s)ds.\]
                          Thus,  integration by parts gives
\begin{align*}
    a_{u,v}
                =\frac{1}{2(u-v)} \Big(h_u(\ell)h^{'}_v(\ell)-h_u(-\ell)h^{'}_v(-\ell)-h_u^{'}(\ell)h_v(\ell)+h_u^{'}(-\ell)h_v(-\ell)\Big)
\end{align*}
 if $u\neq v  $. Note that $h_u$ is odd if $u$ is odd and $h_u$ is even otherwise. So, $h_u h^{'}_v$ is even if $u+v$ is odd and $h_u h^{'}_v$ is odd otherwise. 
Hence, $a_{u,v}=0$ if $u+v$ is odd.  Using the identity $h_u'(s)=sh_u(s)-\sqrt{2u+2}\,h_{u+1}(s) $(\cite[p. 5]{Th93}), we obtain  
\Be\label{auv}
   a_{u,v}
 =\frac{1+(-1)^{u+v}}{\sqrt{2}(u-v)}\big(\sqrt{u+1}\,h_{u+1}(\ell)h_v(\ell)-\sqrt{v+1}\,h_u(\ell)h_{v+1}(\ell)\big).
\Ee
Note that  $\ell-\sqrt{2u+1}\gtrsim \sqrt{\lambda\mu}\sim \ell  $ for $2^{-10}N\mu/d\le u\le 2^{-9}N\mu/d$. Thus,  we have $s^+_{\sqrt{2u+1}}(\ell)\gtrsim \lambda\mu$ as long as $2^{-10}N\mu/d\le u\le 2^{-9}N\mu/d$. By Lemma \ref{S4-asym}  it now follows that 
\[|h_u(\ell)|\lesssim e^{-s^+_{\sqrt{2u+1}}(\ell)} \lesssim e^{-c\lambda\mu} \]
for some $c>0$ if $2^{-10}N\mu/d\le u\le  2^{-9}N\mu/d$.
Combining this with \eqref{auv} and \cite[Lemma 1.5.2]{Th93},  we have $
    |a_{u,v}| \lesssim e^{-c\lambda\mu} $
 for $2^{-10}N\mu/d\le u\le  2^{-9}N\mu/d$ and $v\le N$ if $u\neq v$.  Note that  $ |a_{u,v}| \lesssim 1$ and $|a^\ast_{u,v}| \lesssim \mu^{1/2}$ for any $u,v$. 
 We also note that there is at least one $j\in \{2, \dots, d\}$ such that  $\alpha_j\neq \beta_j$ if $\alpha\neq \beta$, $\alpha\in J$, and $|\beta|=N$. Therefore,   we 
get \eqref{auv-final} because $A_{\alpha,\beta} = a^\ast_{\alpha_{1},\beta_{1}}\prod_{i=2}^d a_{\alpha_{i},\beta_{i}}$. 
\end{proof}

\subsection*{Acknowledgements}
This work was supported by research funds for newly appointed professors of Jeonbuk National University in 2021 and the NRF (Republic of Korea) grants no. 2020R1F1A1A01048520 (E. Jeong), 2022R1A4A1018904 (S. Lee), and a KIAS Individual Grant (MG087001) at Korea Institute for Advanced Study (J. Ryu).

\end{document}